\documentclass[11pt,oneside]{amsart}

\usepackage[english]{babel}    
\usepackage[utf8]{inputenc}    
\usepackage[T1]{fontenc}       
\usepackage{amssymb} 		    
\usepackage{mathtools}         
\usepackage{icomma}            
\usepackage{units}             
\usepackage{enumerate}         
\usepackage{array}             
\usepackage[usenames]{color}   
\usepackage{graphicx}          

\usepackage{dsfont}
\usepackage{mathrsfs}
\usepackage{algorithm}
\usepackage{algpseudocode}

\usepackage[scale=.83]{geometry}

\usepackage[bookmarks, colorlinks=true, linkcolor=black, anchorcolor=black,
citecolor=black, filecolor=black, menucolor=black, runcolor=black,
urlcolor=black, pdfencoding=unicode]{hyperref}  

\newtheorem{thm}{Theorem}
\newtheorem{lemma}[thm]{Lemma}
\newtheorem{prop}[thm]{Proposition}
\newtheorem{cor}[thm]{Corollary}
\theoremstyle{remark}
\newtheorem{remark}[thm]{Remark}
\theoremstyle{definition}
\newtheorem{defn}[thm]{Definition}

\newcommand{\RR}{\ensuremath{\mathbb{R}}}

\newcommand{\N}{\ensuremath{\mathbb{N}}}
\newcommand{\Z}{\ensuremath{\mathbb{Z}}}
\newcommand{\B}{\ensuremath{\mathbb{B}}}

\newcommand{\T}{\ensuremath{\mathbb{T}}}
\newcommand{\K}{\ensuremath{\mathbb{K}}}

\newcommand{\E}{\ensuremath{\mathbb{E}}}

\newcommand{\X}{\mathcal{X}}

\newcommand{\leb}{\mathcal{L}eb}

\newenvironment{itemize*}{\vspace{-10pt}\begin{itemize}\setlength{\itemsep}{0pt}\setlength{\parskip}{2pt}}{\end{itemize}}
\newenvironment{enumerate*}{\vspace{-10pt}\begin{enumerate}\setlength{\itemsep}{0pt}\setlength{\parskip}{2pt}}{\end{enumerate}}
\newenvironment{description*}{\vspace{-12pt}\begin{description}\setlength{\itemsep}{0pt}\setlength{\parskip}{2pt}}{\end{description}}

\newcommand{\Sp}{\ensuremath{\mathbb{S}}}
\newcommand{\abs}[1]{\left|#1\right|}
\newcommand{\set}[2]{\left\{#1\ ; \ #2  \right\}}
\renewcommand{\d}{\mathrm{d}}
\newcommand{\ind}{\ensuremath{\mathds{1}}}

\newcommand{\e}[1]{\mathrm{e}_{#1}}
\newcommand{\cadlag}{c\`adl\`ag}
\newcommand{\levy}{L\'evy}

\DeclareMathOperator{\GL}{GL}
\DeclareMathOperator{\Exp}{Exp}
\DeclareMathOperator{\OO}{O}
\DeclareMathOperator{\SO}{SO}
\DeclareMathOperator{\Hol}{Hol}
\DeclareMathOperator{\supp}{supp}

\makeatletter
\@namedef{subjclassname@2020}{%
	\textup{2020} Mathematics Subject Classification}
\makeatother

\title{L\'evy processes on smooth manifolds with a connection}

\author{Aleksandar Mijatovi{\'c}}
\address{Department of Statistics, University of Warwick, \& The Alan Turing Institute, UK}
\email{a.mijatovic@warwick.ac.uk}
\author{Veno Mramor}
\address{Department of Statistics, University of Warwick, \& The Alan Turing Institute, UK}
\email{veno.mramor@warwick.ac.uk}

\thanks{We thank David Applebaum for helpful comments, which greatly improved an earlier version of the paper. This paper forms a part of the PhD thesis of the second author and we would like to thank the examiners Andreas Kyprianou and Kenneth David Elworthy for thorough reading and in-depth comments which greatly improved the contents of the paper.}

\numberwithin{equation}{section}
\numberwithin{thm}{section}

\begin{document}

\begin{abstract}
We define a \levy{} process on a smooth manifold $M$ with a connection as a
	projection of a solution of a Marcus stochastic differential equation
	on a holonomy bundle of $M$, driven by a holonomy-invariant \levy{}
	process on a Euclidean space. On a Riemannian manifold, our definition
	(with Levi-Civita connection) generalizes the
	Eells-Elworthy-Malliavin construction of the Brownian motion and
	extends the class of isotropic \levy{}
	process introduced in Applebaum~and~Estrade~\cite{isotropiclevy}.  On a
	Lie group with a surjective exponential map, our definition (with left-invariant connection)
	coincides with the classical definition of a (left) \levy{} process
	given in terms of its increments.

Our main theorem characterizes the class of \levy{} processes via their
	generators on $M$, generalizing the fact that the Laplace-Beltrami
	operator generates Brownian motion on a Riemannian manifold.  Its proof
	requires a path-wise construction of the stochastic horizontal lift and
	anti-development of a discontinuous semimartingale, leading to a
generalization of  Pontier~and~Estrade~\cite{horizontalLift} to smooth manifolds with non-unique geodesics between distinct points.
\end{abstract}

\maketitle

\section{Introduction}
\levy{} processes on the Euclidean space $\RR^d$ are fundamental and well-studied stochastic processes with \cadlag{} (right continuous with left limits) paths, characterized by the independence and
stationarity of their increments (see e.g.\ monographs~\cite{BertoinLevy,satolevy,kyprianouFluctuationsLevy} and the references therein).  
\levy{} processes have a natural generalization to
Lie groups 
 introduced in \cite{huntgroups} and have been studied extensively over the last half a century (see
monograph~\cite{liao_levy} for a modern account). The definition relies on the group structure of the state space, which allows to define a natural notion of the increment of a process.
Since there is no natural notion of an increment on a Riemannian manifold $M$, \cite{isotropiclevy} uses a geodesics-based approach to
define isotropic \levy{} processes on $M$. However, these two definitions
are not entirely compatible: for example, if $M$ is a $d$-dimensional torus with a flat metric, isotropic \levy{} processes form a strict subclass of
\levy{} processes on $M$, defined via its Lie group structure.
The same phenomenon may be observed on the Euclidean space $\RR^d$ or indeed any Lie group equipped with a left-invariant Riemannian metric.

The main contribution of the present paper is twofold. \textbf{(I)}~By relaxing key restrictions from the definition of isotropic
\levy{} process on Riemannian manifolds from~\cite{isotropiclevy}, we extend the definition of \levy{} processes to smooth manifolds equipped with a
connection. A fundamental property of \levy{} processes is that they jump in the ``same way'' from any point in the state
space. Our generalization is natural, since it is a connection (and not a metric) that is a required geometric
structure for the comparison of jumps along geodesics at different points on the manifold.
Moreover, a canonical connection is often available when there is no canonical metric.
For example the class of \levy{} processes induced by the left-invariant connection on a Lie group (with a surjective exponential map) coincides with the class of the classical (left) \levy{} process.
Furthermore, our definition with the Levi-Civita connection on a Riemannian manifold extends the class of isotropic
\levy{} processes, closing the aforementioned compatibility gap. 
\textbf{(II)}~Our main theorem (Theorem~\ref{thm:charLevyManifolds} below) characterizes the law of a \levy{} process on a smooth manifold with a connection
via its infinitesimal generator, generalizing the classical result that the law of a Brownian motion on a Riemannian
manifold is characterized by the Laplace-Beltrami operator.

Our definition generalizes the celebrated
Eells-Elworthy-Malliavin construction of a Brownian motion on a $d$-dimensional Riemannian manifold, given as a
projection of a horizontal Brownian motion in the orthonormal frame bundle. The horizontal Brownian motion is
defined as a solution of the Stratonovich stochastic differential equation (SDE) on the orthonormal frame bundle, with the integrator a standard Brownian motion
on $\RR^d$ and the coefficients given by the fundamental horizontal vector fields. Consider a
smooth manifold $M$ with a connection and a corresponding holonomy bundle (a subbundle of the frame
bundle of $M$). We define a \levy{} process $X$ on $M$ to be the projection of a horizontal \levy{} process $U$,
defined via an SDE on a holonomy bundle with the integrator a
\levy{} process $Y$ on $\RR^d$ and the coefficients given by the fundamental horizontal vector fields in the tangent bundle of the holonomy bundle (see
Section~\ref{ConsAndResults} below). Since $Y$ is typically discontinuous, we use the theory of
Marcus SDEs~\cite{marcusSDE,genStratonovich,fujiwaraLevyFlowManifolds,ApplebaumKunitaLevyFlow} in order to obtain the
correct transformation properties under the change of
coordinates on a smooth manifold, not possessed by the Stratonovich integral in the presence of jumps.

A \levy{} process on $\RR^d$ is a time-homogeneous Markov process.
In order to ensure that the time-homogeneous Markov property holds for our
process $X$ on $M$, the horizontal process $U$ on the holonomy bundle has to be (a)~time-homogeneous Markov and (b)~invariant under the
natural action of the corresponding holonomy group. Since $U$ is a strong solution of a Marcus SDE with the integrator $Y$, requirement~(a) forces the process $Y$ to be a \levy{} process on $\RR^d$~\cite{JacodProtterMarkov, EstradeMarkov}.
Condition~(b) is equivalent to the generating triplet of $Y$ being invariant under the action of a holonomy group on $\RR^d$ (see~\eqref{cond:invariance} below).

Holonomy bundle and the corresponding holonomy group are not uniquely determined by the connection on a connected smooth manifold $M$. However, all possible choices of holonomy bundles are diffeomorphic and the corresponding
holonomy groups are conjugate Lie subgroups of the group of invertible matrices in $\RR^d\otimes\RR^d$.
Thus, the choice of the holonomy bundle and the corresponding holonomy group does not affect the class of
all \levy{} processes on $M$ described by our definition (see Proposition~\ref{prop:wellDefLevyonman} below).

The characterization of the law of a \levy{} process $X$ via its infinitesimal generator requires a construction of the stochastic horizontal lift $U$ in the holonomy bundle and
the stochastic anti-development $Y$ in $\RR^d$, such that $U$ solves the Marcus SDE driven by $Y$ and is mapped
into $X$ by the bundle projection. The jumps of the solution $U$ of the Marcus SDE are given via the flow of the
horizontal vector fields, making $X$ jump along the geodesics in $M$.
The main technical difficulty in constructing a stochastic horizontal lift $U$ from the generator of $X$
is that a jump of $X$ may be induced by many different jumps of $U$. This issue disappears when a manifold $M$ possesses a unique geodesic
between any two points, as is assumed in \cite{horizontalLift}, since then the structure of the Marcus SDE
implies that the jumps of $X$ and $U$ are in one-to-one correspondence. However, this geometric assumption is quite
restrictive as it excludes for example all compact Riemannian manifolds e.g spheres of any dimension. More precisely, compact Riemannian manifolds are bounded as metric spaces and on the other hand geodesically complete by the Hopf-Rinow theorem, which would imply unboundedness if the unique geodesic assumption were in place.  
To circumvent this issue in a general setting, 
we  first append compatible ``jump data'' to the process $X$. We can then prove that $X$ and the ``jump data'' together
uniquely determine the horizontal lift $U$ and the anti-development $Y$, see Theorem~\ref{thm:uniqueLift} below. 
Moreover, it is possible to construct on an enlarged probability space suitable ``jump data'' so that the anti-development $Y$ of the
process $X$ is precisely a \levy{} process on $\RR^d$ with the generating triplet induced by the generator of $X$. 

On Riemannian manifolds it is natural to consider the Levi-Civita connection which is uniquely determined by
the metric. A holonomy group of this connection is a Lie subgroup of the orthogonal group
implying that our new
definition of \levy{} processes is a genuine generalization of isotropic \levy{} processes defined in
\cite{isotropiclevy} (see Section~\ref{sec:LevyOnRiem} below for details). In particular, since the holonomy group of the flat connection on $\RR^d$ is trivial, our
definition coincides with the classical definition of \levy{} processes on $\RR^d$.
Lie groups do not necessarily possess a natural Riemannian metric, but have a natural choice of a left-invariant
connection. In Section~\ref{levyonlie} below we show that this choice of a connection implies that a \levy{}
process satisfying our definition is a classical (left) \levy{} process on a Lie group. 
If the exponential map of the Lie group is not surjective, then jumps which are
not in the image of the exponential map cannot be expressed via the geodesics of the left-invariant
connection, so (left) \levy{} process having such jumps cannot be expressed via our definition. However, as soon as
all jumps can be expressed in such a way, the (left) \levy{} process can be expressed via our new definition (see Proposition~\ref{prop:reprLevyOnLieViaHor}).
In particular, on the Lie group with surjective exponential map the two notions of \levy{} processes coincide.

The remainder of the paper is structured as follows. In Section~\ref{ConsAndResults} we apply 
the theory of Marcus SDEs on manifolds to construct a \levy{} process on a smooth
manifold with a connection and state our main results. 
In Section~\ref{examples} we discuss examples of \levy{} processes on manifolds with a connection, 
including special cases of
Riemannian manifolds and Lie groups as well as new examples with exceptional holonomy groups. 
The proofs of our main results are given in Section~\ref{proofsandTechnicalresults}.
Some facts from differential geometry used in our analysis are collected in Appendix~\ref{app:connections}.

\section{Construction and characterization of \levy{} processes on manifolds with a connection}
\label{ConsAndResults}

Let $M$ be a connected smooth $d$-dimensional manifold without boundary. 
Throughout we use the following notation: $C^\infty(M)$
denotes the space of smooth real valued functions on $M$, 
$TM$ is the tangent bundle on $M$, $\Gamma(TM)$
denotes the vector fields on $M$ (i.e.\ smooth sections of the
tangent bundle), see~\cite{kobayashinomizu} as a
general reference for differential geometry. 
A \cadlag{} $M$-valued stochastic process $X= (X_t)_{t\in[0,\tau)}$ is a \emph{semimartingale on a stochastic interval $[0,\tau)$}, where $\tau$ is a predictable stopping time if
$f(X^{\tau_n})$ is a real-valued semimartingale for any $f\in C^\infty(M)$
and $n\in \N,$ where $X^{\tau_n}=(X^{\tau_n}_t)_{t\in\RR_+}$ with $X^{\tau_n}_t=X_{\min\{\tau_n,t\} }$ for $t\in\RR_+$ denotes a stopped process and stopping times
$\set{\tau_n}{n\in \N}$ form an announcing sequence of $\tau,$ i.e.\ $\tau_n \uparrow \tau$ as $n\to \infty$ and
$\tau_n < \tau$ a.s.\ for every $n\in \N$. See \cite{ProtterIntegration,JacodShiryaevLimit} as a general reference for standard (real) semimartingales and
predictable stopping times.
One easily checks using It\^o's formula that if $M=\RR^d$, this definition (with $\tau \equiv \infty$) is equivalent to the standard definition of 
a semimartingale.

The corresponding generalisation of \levy{} processes is not as straightforward
since on a general smooth manifold there is no suitable notion of an increment.
An alternative approach, also used in~\cite{isotropiclevy}, 
exploits the fact that points of a smooth manifold possess tangent
spaces, which are real vector spaces 
providing a local linearization of the manifold.  Each tangent space is
isomorphic to $\RR^d$, making it feasible  to ``develop'' a \levy{} process on a tangent space onto the manifold. 
Since at different points in $M$ the 
tangent spaces are not 
canonically isomorphic,  for such a construction to work we
need to be able to use curves in $M$ to ``connect \&  compare'' different tangent spaces. This highlights the need  for
a (linear) connection on $M$. Our main aim is to show that this geometric structure  is 
sufficient for a \levy{} process on $M$ to be defined, circumventing the need for a Riemannian metric used in~\cite{isotropiclevy}.
There is an intermediate step -- we
need to consider a horizontal process on
the holonomy bundle, which is a solution of a specific Marcus SDE.

In Section~\ref{sec:MarcusSDE} we recall basic properties of Marcus SDEs on manifolds.
In Section~\ref{generalLevy} we apply  Marcus
SDEs to define a horizontal \levy{} process on the holonomy bundle, which is then projected to a \levy{} process on a manifold with a connection.
Section~\ref{sec:viagenerator} describes our class of \levy{} processes via generators on $M$.

\subsection{Marcus SDE} \label{sec:MarcusSDE}
The Stratonovich integral which is used for SDEs on smooth manifolds with continuous semimartingale integrators does not obey the classical chain rule formula when the integrator has jumps.
The first theory of SDEs with jumps on manifolds was introduced in~\cite{rogersonSDEs} and allowed integration against Brownian
motion and Poisson random measures with quite general jump coefficients.
A more modern theory with the same class of integrators is studied in~\cite{KunitaBookStochFlows} and also deals with stochastic flows of diffeomorphisms. 
Technically, such an approach with a suitable choice of a jump coefficient could be used for our
Definition~\ref{def:levyonman} of \levy{} processes on manifold where the integrator is a Euclidean \levy{} process.
However, we also wish to consider more general \cadlag{} semimartingale integrators, especially in the context of stochastic horizontal
lifts and stochastic anti-developments (see Definition~\ref{def:horprocess} and Theorem~\ref{thm:uniqueLift}).

Hence, we will consider a unified approach and utilise a more general notion of a Marcus (or canonical) SDE which was
introduced in the Euclidean spaces in~\cite{marcusSDE} and was later generalized (still in Euclidean
spaces) in~\cite{genStratonovich} and allows integration against arbitrary \cadlag{} semimartingales. Due to correct change-of-variable properties
Marcus SDE arises naturally in the context of smooth
manifolds~\cite{fujiwaraLevyFlowManifolds,ApplebaumKunitaLevyFlow,isotropiclevy}.
In order to describe the generalisation of Marcus SDE 
to smooth manifolds required in this paper, 
we first  recall some notions from differential geometry.

A smooth curve $\gamma\colon I\subseteq \RR \to M$ is an \emph{integral curve of a vector field $\X\in\Gamma(TM)$} if $\dot{\gamma}_s:=d\gamma_s\left(\left.\frac{d}{dt}\right|_s\right)=\X_{\gamma_s}$ holds
for every $s\in I$, i.e.\ the velocity vector of the curve is equal to the value of the vector field $\X$ at the same point. As a consequence of the local existence
and uniqueness theorem for systems of ordinary differential equations (ODEs), there exists a unique integral curve of $\X$ on the maximal time interval once the initial point is fixed.
We say that a vector field is complete if all its integral curves are defined for all times in
$\RR$. This allows us to consider the flow of a vector field and a flow exponential map.  A \emph{flow of a complete vector field $\X$}
is a one-parameter group of diffeomorphisms $\Phi^\X_s\colon M\to M$ for any $s\in\RR$, such that for any $p\in M$ the curve $s\mapsto \gamma_\X^p(s):=\Phi^\X_s(p)$ is a unique integral curve of $\X$ with $\gamma_\X^p(0)=p.$ It
is easy to see that we have the flow property $\Phi^\X_s\circ\Phi^\X_t=\Phi^\X_{s+t}$ for any $s,t\in\RR.$ We can also define the \emph{flow exponential map} $\exp(\X):=\Phi^\X_1\colon M \to M$ which represents a
diffeomorphism which moves points a unit time along the integral curves of a vector field $\X$.

\begin{defn} \label{def:MarcusSDE}
	Let $M$ be a smooth manifold, let $V_1,\ldots,V_\ell\in \Gamma(TM)$ be smooth vector fields on $M$ and let $W$ be an $\RR^\ell$-valued \cadlag{} semimartingale. We call an $M$-valued semimartingale
	X defined on a stochastic interval $[0,\tau)$ \emph{a solution of a Marcus SDE }
	\begin{equation} \label{def:solSDEgen}
		\d X_t= V_i(X_{t-})\diamond \d W^i_t
	\end{equation} if for any $f\in C^{\infty}(M)$ it satisfies the equation
	\begin{equation} \label{marcusDefProp}
		f(X_t)-f(X_0)=\int_0^tV_i f(X_{s-})\circ\d W^i_s + \sum_{0<s\le t} \Big( f\big(\exp(V_i\Delta W^i_s)(X_{s-})\big)-f(X_{s-})-V_if(X_{s-})\Delta W^i_s\Big)
	\end{equation}
	for  $0\le t<\tau,$
	where It\^o's circle $\circ$ denotes the Stratonovich integral, $\Delta W_s=W_s-W_{s-}$ and $\exp$ is a flow exponential map.
\end{defn}
\begin{remark}
	\begin{enumerate}[(i)]
		\item  In~\eqref{def:solSDEgen} and \eqref{marcusDefProp}, as well as in the rest of the paper, we always use Einstein's convention of
		summing over repeated indices, when they appear once as an upper and once as a lower index.
		\item A more precise meaning of~\eqref{marcusDefProp} is as follows. Consider an announcing sequence $\set{\tau_n}{n\in \N}$ of $\tau.$
		We require that for every $n\in \N$ the stopped process
		$X^{\tau_n}$ is a semimartingale and that equation~\eqref{marcusDefProp} with $X$ exchanged by
		$X^{\tau_n}$ holds for every $t\in \RR_+$.
		\item The sum in the defining property~\eqref{marcusDefProp} of a Marcus SDE 
			can be interpreted as follows:  
			the process $X$ jumps only if the
			integrator $W$ jumps and then the jump of $X$ occurs
			along an integral curve of a certain vector field,
			constructed as a linear combination of the driving vector fields with the coefficients given by 
			the components of the jump of the integrator. 
	\end{enumerate}

\end{remark}

\begin{prop} \label{prop:MarcusWellDef}
	Suppose that for every $a=a^i\e{i}\in\RR^d$ the vector field $a^iV_i$ is complete. Then there exists a unique solution $X$ of Marcus
	SDE~\eqref{def:solSDEgen}
	and it is a semimartingale defined on a maximal stochastic interval $[0,\tau),$ where $\tau$ is a predictable stopping time which is an
	explosion time\footnote{If we consider a one-point compactification $\widehat{M} := M \cup \{\partial\}$
		this means that on the event $\{\tau<\infty\}$ as $t\to\tau$ we have $X_t \to \partial$ or equivalently $X$ exits all compacts in $M$.}.
\end{prop}

\begin{remark} \label{remark:jumptoinfty}
	A solution process $X$ of Marcus SDE~\eqref{def:solSDEgen} exists even when vector fields are not complete.
	However, then the stopping time $\tau$ is not necessarily predictable and it might happen that $X$ ``jumps'' to the cemetery point $\partial$ and does not reach it continuously.  To avoid these issues we will insist on the completeness assumption.
\end{remark}
Since all vector fields on compact manifolds are complete, we immediately obtain the following corollary.
\begin{cor}
	If $M$ is compact, then the solution $X$ of Marcus SDE~\eqref{def:solSDEgen} is a semimartingale defined on the whole interval $[0,\infty).$
\end{cor}
A more general type of SDEs on manifolds, which includes~\eqref{def:solSDEgen}, was considered
in~\cite{CohenSDEs1,CohenSDEs2}. 
But for the sake of completeness and due to its simplicity, we give a direct proof of Proposition~\ref{prop:MarcusWellDef} 
in Section~\ref{sec:MarcusSDEResults} below.

\subsection{Horizontal \levy{} processes and \levy{} processes on a smooth manifold with a connection} \label{generalLevy}

Let $F(M)$ denote a frame bundle over $M$ (defined as a bundle of linear frames in \cite[p.~56]{kobayashinomizu}) with a bundle projection $\pi\colon F(M)\to M$. The bundle $F(M)$ is a
smooth manifold of dimension $d+d^2$ and the projection $\pi$ is a smooth map. A frame $u\in F(M)$, such that $\pi(u)=p$, is an ordered basis $u_1,\ldots,u_d$ of the tangent space $T_pM$.
We may regard it as a linear isomorphism $u\colon \RR^d\to T_pM$ given by 
\begin{equation} \label{def:framefunction}
	u(x)=x^iu_i \quad \text{for each} \quad x=x^ie_i\in \RR^d,
\end{equation}
where $e_1,\ldots,e_d $ is a standard  basis of $\RR^d.$
The general linear group $\GL(d)$, consisting of invertible matrices in
$\RR^d\otimes \RR^d$, acts on $F(M)$ on the right by $R_g(u)=ug:=u\circ g$,
where $g\in\GL(d)$ is considered as an automorphism of $\RR^d.$ This turns
$F(M)$ into a principal fibre bundle (over $M$) with a structure group $\GL(d)$
(see definition in \cite[\S5,~Ch.~I]{kobayashinomizu}).

We assume that $M$ is equipped with a (linear) connection. A connection is
given by a covariant derivative which prescribes how vector fields can be
derived along other vector fields. A covariant derivative induces a direct sum
splitting $T_uF(M)=H_uF(M)\oplus V_uF(M)$ of the tangent spaces of the frame
bundle $F(M)$ at each frame $u$. Whereas vertical spaces $V_uF(M):=\ker d\pi_u$
are canonically defined, the horizontal spaces
$H_uF(M)$ are induced by the covariant derivative and they are invariant by the differential of a right group action of $\GL(d)$.
This (invariant) splitting is what is classically defined as a (linear)
connection in the frame bundle $F(M)$ in \cite[Chs.~II and
III]{kobayashinomizu}. It is well-known that such an  invariant 
splitting on $F(M)$ is equivalent to  a covariant derivative, see~\cite[Thm.~7.5 in Ch.~III]{kobayashinomizu}. More details about
covariant derivatives and the induced splitting are given in
Appendix~\ref{app:connections}.

The horizontal vector spaces are of dimension $d$. Since 
$d\pi_u$ 
is surjective, the restricted map $d\pi_u\colon H_uF(M)\to T_{\pi(u)}M$
is an isomorphism, inducing the horizontal lift 
in $T_uF(M)$ of any tangent
vector in $T_{\pi(u)}M$. Thus, for each $x\in\RR^d$ there exists a
canonical fundamental horizontal vector field $H_x$ on $F(M)$
(defined as a
standard horizontal vector field in \cite[p.~119]{kobayashinomizu}), uniquely determined by:
\begin{enumerate}[(H1)]
	\item $H_x(u)\in H_uF(M),$
	\item $d\pi_u(H_x(u))=u(x)$ \label{def:horVFChar}
\end{enumerate}
for each $u\in F(M).$
It follows that $x\mapsto H_x$ is a linear map and we denote $H_i:=H_{e_i}$ for $e_1,\ldots,e_d $ the standard  basis of $\RR^d.$ 
In order to be able to use Proposition~\ref{prop:MarcusWellDef} and avoid the non-predictability issues outlined Remark~\ref{remark:jumptoinfty}, we will 
assume that all the vector fields $H_e$,~$e \in \RR^d$ are complete, which is equivalent to geodesic completeness of the manifold $M$ (see paragraph
preceding equation~\eqref{bothExponentials} for relation between geodesics and integral curves of fundamental
horizontal vector fields).

\begin{defn} \label{def:horLevy}
	Let $M$ be equipped with a connection which turns it into a geodesically complete manifold.
	An $F(M)$-valued semimartingale $U$ defined on a stochastic interval $[0,\tau)$ is called a \emph{horizontal \levy{} process} if there exists an
	$\RR^d$-valued \levy{} process $Y$ (see \cite[Def.~1.6]{satolevy} for a definition of $\RR^d$-valued \levy{} processes) such that $U$ is a solution of a Marcus SDE 
	\begin{equation} \label{eq:horlevy}
		\d U_t=H_i(U_{t-})\diamond \d Y^i_t,\qquad U_0=u\in F(M),
	\end{equation}
	on the maximal (stochastic) interval $[0,\tau).$
\end{defn}
\begin{remark}
	\begin{enumerate}[(i)]
		\item Horizontal \levy{} processes were originally introduced in~\cite{ApplebaumHorizontalFirst} on the orthonormal bundle of the Riemannian manifold.
		\item We require geodesic completeness so that the flow exponential map is well defined on complete vector fields $H_x, x \in \RR^d$.
		Proposition~\ref{prop:MarcusWellDef} then shows that a horizontal \levy{}
		process $U$ is uniquely determined by the \levy{} process $Y$
		and is defined on the maximal stochastic interval
		$[0,\tau),$ where the assumption on geodesic completeness guarantees that $\tau$ is predictable.
	\end{enumerate}
\end{remark}

In the next step, we use the bundle projection $\pi$ and consider the projected
process $X:=\pi(U)$ which is a candidate for a \levy{} process on $M$. Since
the integrator $Y$ of Marcus SDE~\eqref{eq:horlevy} is a \levy{} process, the
horizontal \levy{} process $U$ is always a strong
Markov process \cite[Thm.~5.1]{genStratonovich}, but $X$ in general need not be. 
More precisely,  if $X_0=p$ we only know that $U_0$ is a frame in the fibre above $p$, 
i.e.\  $U_0=u$ such that $\pi(u)=p.$ But $U_0$ might be a different frame $v$ in the same fibre. 
The future evolution of $U$, and hence of $X$, could be vastly different in these two cases, $U_0$ = $u$ or $v$. 
However, since $u$ and $v$ are in the same fibre, there exists $g\in \GL(d)$ such that $u=vg$.
With this in mind, consider 
the process $\widetilde{U}:=Ug$, which by 
Lemma~\ref{lemma:transformMarcusSDE} below solves 
the Marcus SDE 
$$\d \widetilde{U}_t=(R_g)_*H_i(\widetilde{U}_{t-})\diamond \d Y^i_t,$$
where push-forward vector fields $(R_g)_*H_i$ are given in \eqref{def:pushforwardVF} below. Moreover, by Lemma~\ref{lemma:changeHorizontal} in the appendix, we have
$(R_g)_*H_x=H_{g^{-1}x}$ for every $x\in \RR^d$. Thus, linearity of $x\mapsto H_x$ and Lemma~\ref{lemma:linearchangeMarcusSDE} imply that $\widetilde{U}$ is also a solution of Marcus SDE
$$\d \widetilde{U}_t=H_i(\widetilde{U}_{t-})\diamond \d \widetilde{Y}^i_t$$ with $\widetilde{U}_0=ug=v$ where
$\widetilde{Y}:=g^{-1}Y.$ 
Consider a solution $V$ 
of Marcus SDE~\eqref{eq:horlevy}
with initial condition $V_0$, which 
is driven by $Y$.  Therefore, if $Y$ and
$\widetilde{Y}$ were to have the same law, then (weak) uniqueness of solutions of
Marcus SDEs shows that $\widetilde{U}$ and $V$ also have the same law, thus implying that 
the processes $\pi(U)=\pi(\widetilde{U})$
and $\pi(V)$ are equal in law. Hence by assuming the appropriate invariance of the integrator $Y$ we may use the result of Dynkin on transformations of Markov processes~\cite[Thm.~10.13]{DynkinBook} which ensures the Markovianity of the projected process $X$.
Fortunately, as we will see below, we do not have to check such an invariance for every $g\in \GL(d),$ but only for elements of a specific subgroup -- the holonomy group, which we now describe.

For any frame $u\in F(M)$, a connection on $M$ allows us to reduce the frame bundle $F(M)$ to a subbundle
$P(u)$ known as the holonomy bundle. The \emph{holonomy bundle $P(u)$} is a principal
fibre bundle over $M$ with a structure group given by the \emph{holonomy group $\Hol(u)$}, a
Lie subgroup of $\GL(d)$ \cite[Thm.~4.2 in Ch.~II]{kobayashinomizu}.  More
precisely,  a holonomy bundle $P(u)$
consists of all the frames in $F(M)$ connected to $u$ via horizontal curves in $F(M)$ (i.e.\ 
curves with velocity vectors in horizontal spaces). 
The essential property for our definition of \levy{} processes is that
the fundamental horizontal vector fields $H_x$, $x\in \RR^d$, are tangent to the
subbundle $P(u)$ (this follows by the definition of the holonomy bundle via
horizontal curves), hence they can be interpreted as vector fields on $P(u)$.
By Proposition~\ref{prop:MarcusWellDef}, any horizontal \levy{}
process $U$ (a solution of Marcus SDE~\eqref{eq:horlevy}) with $U_0\in P(u)$
stays in $P(u)$ for all times for which it is defined.Hence to use~\cite[Thm~10.13]{DynkinBook} which ensures that the
projection $X=\pi(U)$ is Markov, it is sufficient to check that the integrator $Y$
is $\Hol(u)-$invariant, i.e.\ the processes $Y$ and $gY$ are equal in law for every $g\in \Hol(u)$
or equivalently the characteristics of $Y$ satisfy~\eqref{cond:invariance} below.
\begin{defn} \label{def:levyonman}
		Let $Y$ be a $\Hol(u)$-invariant \levy{} process on $\RR^d$ and let $U=(U_t)_{t\in[0,\tau)}$ be a horizontal \levy{} process on the holonomy bundle $P(u)$ that
		solves Marcus SDE~\eqref{eq:horlevy} with $U_0\in P(u)$ a.s. Then with $\pi$ the bundle projection as above, $X=\pi(U)$ is said to be a \emph{\levy{} process on $M$}.
\end{defn}

\begin{prop} \label{prop:wellDefLevyonman}
A process $X$ on $M$ satisfying Definition~\ref{def:levyonman} is a Markov process and is a semimartingale on $[0,\tau)$.
The class of all \levy{} processes on a connected smooth manifold $M$
is invariant under the choice of the frame $u\in F(M)$ and the corresponding holonomy
bundle $P(u)$.  \end{prop}

\begin{remark} \label{remark:levyonman}
	\begin{enumerate}[(i)]
		\item The holonomy bundle is a reduction of the frame bundle (see \cite{kobayashinomizu} for definition of a reduction). In fact, it is the smallest possible reduction compatible with the connection \cite[Remark~1.7.14]{DifGoMathPhys}. Hence we
		cannot in general reduce the holonomy group invariance of the \levy{} integrator any further and the class of \levy{} processes on the smooth manifold $M$ with a connection from Definition~\ref{def:levyonman} is the largest possible. 
		\item Due to a particular form of Marcus SDE~\eqref{eq:horlevy}, we know that $U_t=\exp(H_i\Delta Y^i_t)(U_{t-})$ holds for any $t<\tau,$ where $\exp$ is a flow
		exponential map (see Section~\ref{sec:MarcusSDE}). Then equation~\eqref{bothExponentials} shows that $X_t=\Exp_{X_{t-}}(U_{t-}\Delta Y_t)$ holds for any $t<\tau,$ where $\Exp_p$
		denotes a geodesic exponential map based at a point $p\in M$ (see Appendix~\ref{app:connections}). Therefore, jumps of $X$ occur only along geodesics.
		\item In the special case of Riemannian manifolds, \cite[Sec.~3]{mohariReduction} considered a holonomy
		bundle reduction and obtained similar results to Proposition~\ref{prop:wellDefLevyonman}. 
		More details about the Riemannian case are given in Section~\ref{sec:LevyOnRiem}.
		\item Note that if a \levy{} process $Y$ is adapted to filtration $(\mathcal{F}_t)_{t\in\RR_+}$,
		then also the horizontal \levy{} process $U$ and the \levy{} process $X$ on $M$ are adapted to the same filtration.
		However, even when we consider the natural filtration $(\mathcal{F}^Y_t)_{t\in\RR_+}$ of $Y$,
		it could happen that the natural filtrations of $U$ and $X$ are strictly smaller than $(\mathcal{F}^Y_t)_{t\in\RR_+}$.
	\end{enumerate}
\end{remark}

\subsection{Generators of \levy{} processes on manifolds with a connection} \label{sec:viagenerator}

In this section we give a characterisation via generators of  \levy{}
processes on smooth manifolds. 
The first step is to analyse the generators
of horizontal \levy{} processes.

\begin{lemma} \label{lemma:genhorlevy}
	The (extended) generator of a horizontal \levy{} process $U$, which solves  Marcus SDE~\eqref{eq:horlevy}, is given by
	\begin{equation} \label{eq:generatorHorLevy}
	\mathscr{L}_U(f)(v)=H_b f(v)+\frac{1}{2}a^{ij}H_i H_j f(v)+\int_{\RR^d\backslash\{0\}}\left( f(\exp(H_x)(v))-f(v)-\ind(\abs{x}<1)x^iH_i f(v)\right) \nu(\d x)	
	\end{equation} 
	for any frame $v\in P(u)$, 
	where $(a,\nu,b)$ is a generating triplet~\cite[p.~65]{satolevy} of a \levy{} integrator $Y$ associated with a cut-off function $\chi(x)=\ind(\abs{x}<1)$
	and the domain of $\mathscr{L}_U$ includes smooth functions on the holonomy bundle $P(u)$. 
\end{lemma} 

Recall that a \levy{} process $Y$ is $\Hol(u)$-invariant if the laws of processes $Y$ and $gY$ are the same for every $g\in\Hol(u)$. This holds if and only if the generating triplet
$(a,\nu,b)$ of $Y$ satisfies
\begin{equation} \label{cond:invariance}
\nu\text{ is }g\text{-invariant (i.e.\ }g_*\nu=\nu); \quad gag^\top=a;\quad    gb=b+\int_{\RR^d\backslash\{0\}}\left(\ind(\abs{g^{-1}x}<1)-\ind(\abs{x}<1)\right)\nu(\d x)
\end{equation}
for every $g\in \Hol(u)$, where the push-forward measure $g_*\nu$ is defined by $g_*\nu(A):=\nu(g^{-1}(A))$ for every Borel measurable 
set $A\in\mathfrak{B}(\RR^d)$, $a$ and $g$ are interpreted as a matrices in $\RR^d\otimes\RR^d$ and $^\top$ denotes matrix transpose.
Note that when $\Hol(u)$ is a subgroup of the orthogonal group $\OO(d),$ the condition on the drift reduces
to $gb=b.$ 

Since \levy{} process $X$ is given as a projection of a horizontal \levy{} process $U$, its
generator is a ``push-forward'' of a generator~\eqref{eq:generatorHorLevy} and invariance
conditions~\eqref{cond:invariance} ensure it is indeed a well-defined operator. 

\begin{prop} \label{prop:genX}
	The generator $\mathscr{L}_X$ of a \levy{} process $X$ on a smooth manifold $M$ with a connection is given by
	\begin{align} \label{eq:generatorX}
		\mathscr{L}_X(f)(p)&=H_b (f\circ \pi)(v)+\frac{1}{2}a^{ij}H_iH_j(f\circ \pi)(v)\\
		&+\int_{\RR^d\backslash\{0\}}\left( f(\Exp_p(vx))-f(p)-\ind(\abs{x}<1)x^iH_i (f\circ \pi)(v)\right) \nu(\d x), 	\nonumber
	\end{align}
	where $v$ is any frame (in the holonomy bundle $P(u)$) with $\pi(v)=p$ and $\Exp_p$ is a geodesic exponential map based at $p$. 
	In particular, \eqref{eq:generatorX} does not depend on the choice of $v$.
	The domain of the generator $\mathscr{L}_X$ comprises functions $f$ on $M$ such that $f\circ\pi$ is in the domain of the generator $\mathscr{L}_U.$
\end{prop}

\begin{remark}
	In special cases the formula in~\eqref{eq:generatorX} can be simplified.
	\begin{enumerate}[(i)] \label{remark:simplifyGenerator}
		\item Suppose that $-I_d\in\Hol(u),$ where $I_d$ is the identity matrix in $\RR^d\otimes\RR^d.$
		Then $-v\in P(u)$, and hence $H_i(f\circ \pi)(-v)=-H_i(f\circ \pi)(v)$,  and~\eqref{cond:invariance} implies $b=0.$ Assume further that $a=0.$ Then by~\eqref{eq:generatorX} applied with $v$ and $-v$ yields
		\begin{align}
			\mathscr{L}_X(f)(p)&=\frac{1}{2}\int_{\RR^d\backslash\{0\}}\left( f(\Exp_p(vx))-f(p)-\ind(\abs{x}<1)x^iH_i (f\circ \pi)(v)\right) \nu(\d x) 
			\nonumber \\
			&+\frac{1}{2}\int_{\RR^d\backslash\{0\}}\left( f(\Exp_p(-vx))-f(p)+\ind(\abs{x}<1)x^iH_i (f\circ \pi)(v)\right) \nu(\d x)\nonumber \\
			&=\frac{1}{2}\int_{T_pM}\left(f(\Exp_p(y))-2f(p)+f(\Exp_p(-y))\right)\nu_p(\d y), \label{eq:expressJumpGenerator}
		\end{align}
	where the push-forward measure $\nu_p:=v_*\nu$ is independent of $v$ by~\eqref{cond:invariance}, implying the expression for the jump part of the generator obtained in~\cite{ApplebaumRosieL2}.
	There are many examples of Riemannian manifolds with a Levi-Civita connection where $-I_d\in\Hol(u),$ e.g.\ when $d$ is even and $\Hol(u)=\SO(d)$ is the special orthogonal group.  
	There also exist manifolds with a connection not induced by the Riemannian metric such that $-I_d\in\Hol(u),$ see Section~\ref{section:cylinder}.
	\item In the case the stochastic anti-development \levy{} process $Y$ is $\OO(d)$-invariant
	in the sense of~\eqref{cond:invariance}, the generating triplet equals $(\alpha I_d,\nu,0)$ for some $\alpha\ge 0$ and an $\OO(d)$-invariant \levy{} measure $\nu$. 
	If $M$ is a Riemannian manifold equipped with the Levi-Civita connection, the holonomy bundle is a subbundle of the orthonormal frame bundle and the horizontal
	Laplace operator satisfies $\sum_{i=1}^dH_i^2(f\circ\pi)(v)=\Delta_Mf(p)$ for any orthonormal frame $v$ in the fibre above $p$,
	where $\Delta_M$ is the Laplace-Beltrami operator on $M$. Thus~\eqref{eq:generatorX} and~\eqref{eq:expressJumpGenerator} imply
	$$\mathscr{L}_X(f)(p)=\frac{\alpha}{2}\Delta_Mf(p)+\frac{1}{2}\int_{T_pM}\left(f(\Exp_p(y))-2f(p)+f(\Exp_p(-y))\right)\nu_p(\d y),$$
	which equals the generator of an isotropic \levy{} process in~\cite{isotropiclevy}, where the improper integrals in the jump part were used implicitly in \cite[Eq.~(3.6)]{isotropiclevy}.
	\end{enumerate}
\end{remark}

For the formula~\eqref{eq:generatorX} to be of use, 
it is crucial to prove that any Markov process $X$ on $M$ with the generator 
$\mathscr{L}_X$ is indeed a \levy{} process on $M$ satisfying Definition~\ref{def:levyonman}.  
Put differently,
we need to construct an $\RR^d$-valued \levy{} process $Y$ such that $X$ can be represented via the Marcus SDE in Definition~\ref{def:levyonman}.
Typically, we will not be able to construct $Y$ on the same probability space on which $X$ is defined and we will need to consider an extended probability space.
Additionally, the filtration
$(\mathcal{H}_t)_{t\in \RR_+}$, to which $X$ is adapted, will also have to be extended to a larger filtration
$(\mathcal{G}_t)_{t\in \RR_+}$ and this will constitute a strong Markov extension (see~\cite[Def.~2.47]{CJrepr}).
Existence of the appropriate process $Y$ on the extended probability space is the content of the following theorem.

\begin{thm} \label{thm:charLevyManifolds}
	Let $X$ be a Markov process on $M$ adapted to $(\mathcal{H}_t)_{t\in \RR_+}$ with a generator $\mathscr{L}_X$ given by~\eqref{eq:generatorX}. Then there exists an extended probability space and
	on it an $\RR^d$-valued \levy{} process $Y$ adapted to a an extended filtration $(\mathcal{G}_t)_{t\in \RR_+}$ which
	forms a strong Markov extension, such that $X=\pi(U)$ where $U$ is a horizontal \levy{}
	process solving Marcus SDE~\eqref{eq:horlevy}.
\end{thm}

\subsection{Construction of horizontal lifts and anti-developments for discontinuous processes} \label{generalHorLift}

As mentioned above, the proof of Theorem~\ref{thm:charLevyManifolds} requires
us to construct an integrator of Marcus SDE~\eqref{eq:horlevy}, such that the
solution of this Marcus SDE is projected into a given process. In order to highlight the difficulties arising due
to the presence of jumps, it is natural to consider such
a construction in a more general setting
of horizontal lifts and stochastic anti-developments of general
semimartingales. 
Theorem~\ref{thm:uniqueLift} below, which resolves these problems and we hope is of independent interest, 
is a key step in the proof of Theorem~\ref{thm:charLevyManifolds}.

\begin{defn} \label{def:horprocess}
	An $F(M)$-valued semimartingale $U$ on $[0,\tau)$ is said to be \emph{horizontal} if there exists an
	$\RR^d$-valued semimartingale $W=(W_t)_{t\in\RR_+}$ with $W_0=0,$ such that $U$ is a solution of a Marcus SDE 
	\begin{equation}\label{def:horDisCont}
	\d U_t=H_i(U_{t-})\diamond \d W^i_t
	\end{equation}
	on the stochastic interval $[0,\tau).$
	The semimartingale $W$ is then called a \emph{stochastic anti-development} of $U$ (or of its projection $X=\pi(U)$).  Additionally, if $X$ is an $M$-valued semimartingale on $[0,\tau)$, then a horizontal
	semimartingale $U$ is called a \emph{stochastic horizontal lift} of $X$ if $\pi(U)=X.$
\end{defn}

\begin{remark}
	\begin{enumerate}[(i)]
		\item If $W$ in Definition~\ref{def:horprocess} is in addition  a \levy{} process on $\RR^d$, we obtain Definition~\ref{def:horLevy}.
		\item Only times up to $\tau$ are used in Marcus SDE~\eqref{def:horDisCont} so it suffices for $W$ to be semimartingale defined only on $[0,\tau).$
	\end{enumerate}
\end{remark}

If $X$ is a continuous $M$-valued semimartingale on $[0,\tau),$ then 
the stochastic anti-development and the stochastic horizontal lift of $X$ always exist and are unique (on $[0,\tau)$) once the initial
frame $U_0$ satisfying $\pi(U_0)=X_0$ has been chosen \cite[Sec.~2.3]{hsumanifolds}.
However, in the discontinuous case, the uniqueness of the stochastic anti-development and the existence of the stochastic
horizontal lift may fail.  
To illustrate some potential issues 
posed by jumps, we first focus on the stochastic anti-development of a horizontal semimartingale $U$, which exists by definition. Uniqueness, however, is related to whether
at a jump time $s$ of $U$ there exists a unique fundamental horizontal vector field (and hence a unique jump of the integrator corresponding to it), whose flow exponential map connects
the two endpoints $U_{s-}$ and $U_s$ of the jump. A simple example where uniqueness fails is when $M=\Sp^1:=\set{z\in \RR^2}{\abs{z}=1}$ is a unit circle equipped with the Levi-Civita connection induced by the standard Riemannian metric of the ambient space $\RR^2$.
Then one easily sees that $F(\Sp^1)$ is actually diffeomorphic to $\Sp^1\times
(\RR\backslash\{0\})$ and the fundamental horizontal vector field is
$H_1=\left(\frac{d}{d \theta},0\right).$ Hence $\exp(2\pi k H_1)(u)=u$ for
$k\in\Z$, implying that  no jump occurs in the case where the
integrator has a jump of a size $2\pi k$. We can then take an arbitrary real
semimartingale $W$ and let $U$ be a horizontal semimartingale satisfying $\d
U_t=H_1(U_{t-})\diamond \d W_t$. But then $U$ is also a solution of a Marcus
SDE
$\d U_t=H_1(U_t)\diamond \d \widetilde{W}_t$, where $\widetilde{W}=W+2\pi N$ and $N$ is a Poisson process. In particular, this shows that a continuous horizontal process may
have a discontinuous stochastic anti-development (its continuous stochastic anti-development is unique). 

On the other hand, the stochastic horizontal lift of a semimartingale $X$ need
not exist. Jumps of a horizontal semimartingale occur only along integral
curves of fundamental horizontal vector fields, hence jumps of its projection
can only occur along geodesics i.e.\ the end-point of a jump has to be given as a
point in the image of a geodesic exponential map based at the start point of the
jump. But it could be the case that between the two endpoints $X_{t-}$ and
$X_t$ of the jump
there does not exist a geodesic. 
For example, we can take as our manifold a punctured plane
$\RR^2\backslash\{0\}$, which is an open subspace of the Euclidean plane, and
since there are no geodesics
(straight lines) between $x$ and $-x$, a process making such jumps cannot have
a stochastic horizontal lift.

Furthermore, even when suitable geodesics exist, they might not be unique,
implying the non-uniqueness of horizontal lifts. A typical example is given by
the two antipodal points on a sphere and there exists a whole one parameter
family of possible geodesics (along great circles). Even between two arbitrary
points on a sphere one could travel several times around the sphere along a
great circle before stopping at the endpoint,
thus breaking uniqueness. Another similar issue related to non-uniqueness of
geodesics is that even when we observe no jump of a semimartingale, we could
still sometimes force its horizontal lift to jump in a way that would not manifest itself
as a jump after projection. 

Under suitable assumptions, these issues can be dispensed with. In the setting
of Riemannian manifolds, \cite{horizontalLift} considered 
semimartingales such that there exists a unique geodesic between the two end
points of any jump. They were able to show that for such processes the
stochastic horizontal lift exists and is unique once the initial frame is
fixed, and the same is true for the stochastic anti-development.
As mentioned in the introduction, the assumption of unique geodesics is quite strict, excluding
compact Riemannian manifolds, such as spheres and projective spaces. 
In order to deal with general smooth manifolds we provide an alternative
approach based on Theorem~\ref{thm:uniqueLift} below. 

The main issues when constructing horizontal lift and anti-development of a
semimartingale $X$ concern the jumps of $X$. It is not sufficient to simply
know that a jump occurred at some time $s$, but we also need to know how it
happened, i.e.\ which specific geodesic was used to induce the jump. Geodesics
can be parametrized via a geodesic exponential map, hence we need to specify a
tangent vector $v\in T_{X_{s-}}M$ such that $X_s=\Exp_{X_{s-}}(v)$ holds.
Typically, we cannot uniquely determine such a vector from $X$ alone, hence
this ``jump data''
needs to be appended to the semimartingale $X.$
Note that such ``jump data'' will typically have to be defined on an extended probability space and will be adapted to an
extended filtration $(\mathcal{G}_t)_{t\in\RR_+}$ to which $X$ will still be adapted.

In order to record the necessary ``jump data'',  
fix a holonomy bundle $P(u)$ with a holonomy group $\Hol(u)$ and
consider some (global) section $q\colon M\to P(u)$,
i.e.\ a map $q$ satisfying $\pi\circ q=\mathrm{id}_M$. We require that $q$ is measurable and that $q$ maps any compact set in $M$ into a pre-compact set in $P(u)$.
These requirements would hold trivially if $q$ were smooth, but
in general it is not possible to construct even a continuous section. However, the fibre bundle structure
allows us to find smooth local sections. Combining countably many such local
sections we can always find a measurable (global) section $q$ satisfying compactness condition. Fixing one such section 
allows us to record the ``jump data'' as an $\RR^d$-valued process.
\begin{defn} \label{def:SemiJump}
	An $\RR^d$-valued process $J$ An $\RR^d$-valued process $J$ (possibly) on an extended probability space and adapted to filtration
	$(\mathcal{G}_t)_{t\in\RR_+}$ is said to \emph{represent the jumps of a semimartingale $X$ on $[0,\tau)$} if 
	\begin{equation} 
\label{jumpsemiBound}
		X_s=\Exp_{X_{s-}}(q(X_{s-})J_s)\>\>\text{for } s<\tau \quad 
	\text{		and}\quad
	\sum_{u\le t} \abs{J_u}^2 < \infty\>\> \text{for } t\in\RR_+,
	\end{equation} 
	for a measurable global section $q\colon M \to P(u)$ mapping compact sets in $M$ into pre-compact sets in $P(u)$. 
\end{defn}
 
\begin{remark} \label{remark:defineJ}
	\begin{enumerate}[(i)]
		\item It would suffice for a process $J$ to be defined only on $[0,\tau)$, but it turns out to be convenient in proofs to have it defined for all times in $[0,\infty).$
		\item 
		\label{cond:remAnti} If we start with an anti-development $W$ and use it to construct the processes $U$ and $X$, we may then represent the jumps of $X$ by a process $J_s :=q(X_{s-})^{-1}U_{s-}\Delta W_s.$ Processes $U$ and $X$ are \cadlag{}, hence locally
		bounded, meaning that for every $t\in\RR_+$ and a.e.\ $\omega$ the paths $X(\omega)$ and $U(\omega)$ are included in compact sets for all times in $[0,t]$, hence by compactness condition on $q$ also matrices
		$q(X_{s-}(\omega))^{-1}U_{s-}(\omega)$ are included in a compact set for $s\in[0,t]$, so that these matrices are bounded in operator norm. 
		Since jumps of a semimartingale $W$ are square summable we thus immediately obtain the summability condition in~\eqref{jumpsemiBound}.
		Hence it is natural to require this summability condition for any process $J$ in Definition~\ref{def:SemiJump}.
		Note also that any such $J$ may be different from $0$ for at most countably many times $s$.
		\item For a general semimartingale $X$ there might not exist any process $J$ which represents its jumps. However, \eqref{cond:remAnti} above shows that if $X$ has an anti-development, such
		a process always exists. Actually, Theorem~\ref{thm:uniqueLift} below shows that being able to represent jumps of $X$ by some process $J$ is equivalent to $X$ having a stochastic anti-development.
		The property of $X$ having a stochastic anti-development does not depend on the choice of a section.
		Moreover, this means that if a pair $(q,J)$ satisfies~\eqref{jumpsemiBound}, then for any other measurable section $\widetilde{q}$ satisfying compactness condition we can define a process
		$\widetilde{J}_s:=\widetilde{q}(X_{s-})^{-1}U_{s-}\Delta W_s$ (where $U$ and $W$ are as in \eqref{cond:remAnti}) so that $(\widetilde{q},\widetilde{J})$ satisfies~\eqref{jumpsemiBound}, including the summability condition.
		This shows that for a semimartingale $X$ the property of being able to represent its jumps as in Definition~\ref{def:SemiJump}
		does not depend on the choice of a measurable section satisfying compactness condition.
		Furthermore, the results of Theorem~\ref{thm:uniqueLift} can be obtained for any choice of such section.
	\end{enumerate}
\end{remark}

When constructing a horizontal lift $U$ and anti-development $W$ of $X=(X_t)_{t\in[0,\tau)}$, the process $J$ will enable us to uniquely determine the corresponding jumps of $U$ and $W$. More precisely, by
Definition~\ref{def:SemiJump}, Remark~\ref{remark:defineJ} and equation~\eqref{bothExponentials} we should have

\begin{equation} \label{jumpsforlift}
U_s=\exp\left(H_{U^{-1}_{s-}q(X_{s-})J_s}\right)(U_{s-})
\end{equation} and

\begin{equation} \label{jumpsforanti}
	\Delta W_s=U_{s-}^{-1}q(X_{s-})J_s,
\end{equation} 
for all (jump) times $s<\tau$. 

Thus, we are able to start with $X$ and $J$ and construct a unique horizontal lift and anti-development.

\begin{thm} \label{thm:uniqueLift}
	Let $X$ be a $M$-valued semimartingale defined on $[0,\tau)$ along with a process $J$ which represents its jumps.
	Let $u_0\in\pi^{-1}(\{X_0\})\cap P(u)$.
	Then there exists a unique horizontal lift $U=(U_t)_{t\in[0,\tau)}$ of $X$ with $U_0=u_0$ and values in
	$P(u)$ satisfying \eqref{jumpsforlift}. It has a unique anti-development $W=(W_t)_{t\in[0,\tau)}$ satisfying \eqref{jumpsforanti}.
\end{thm}
\begin{remark}
	Both $U$ and $W$ are adapted to filtration $(\mathcal{G}_t)_{t\in\RR_+}$ to which $X$ and $J$ were adapted to start with.
\end{remark}

As a corollary we may recover the results of \cite[Thm.~3.2 and Prop.~4.3]{horizontalLift}. When there exists a unique geodesic between any two endpoints of a jump of a semimartingale $X$, then there exists a
unique\footnote{Process $J$ is uniquely determined only on $[0,\tau)$, but those are the only values used in
	construction of horizontal lift and anti-development in Theorem~\ref{thm:uniqueLift}} process $J$ which represents its
jumps and hence we obtain a unique horizontal lift and a unique anti-development.

\section{Examples} \label{examples}
In this section we show how Definition~\ref{def:levyonman} relates to previous definitions of \levy{} processes on Riemannian manifolds and on Lie groups. We also provide some concrete examples which illustrate different definitions and (slight) differences between them.

\subsection{Riemannian manifolds} \label{sec:LevyOnRiem}

On a Riemannian manifold $M$ there exists a canonical Levi-Civita connection uniquely determined by the metric.
It is easy to see that the associated fundamental horizontal vector fields are then also
vector fields on the orthonormal frame bundle $O(M)$ which consists of orthonormal frames $u\colon\RR^d\to T_{\pi(u)}M$, i.e.\ frames which are isometries. Thus, the horizontal \levy{} process -- a solution of a Marcus
SDE~\eqref{eq:horlevy} -- stays on the orthonormal frame bundle if it is started there. In \cite{isotropiclevy} they consider such horizontal \levy{}
processes and then their projections to $M$ are Markov if the integrator $Y$ is a \levy{} process invariant by the full orthogonal group $\OO(d)$. In this case they call the projections isotropic \levy{} processes on the Riemannian manifold $M$. 
 
However, it is often possible to reduce the orthonormal frame bundle even further, and hence require less invariance. In the Riemannian case such a reduction  was first considered in \cite[Sec.~3]{mohariReduction}. Namely, compatibility of the Levi-Civita connection with the metric means 
that if $u$ is some (any) fixed orthonormal frame, then the holonomy bundle $P(u)$ is a subbundle of $O(M)$ and the holonomy group $\Hol(u)$ is a Lie subgroup of $\OO(d)$. Thus, for a fixed orthonormal frame $u$ the
\levy{} processes from Definition~\ref{def:levyonman} require less invariance of the integrator than isotropic \levy{} processes, and hence the latter are just a special case of the former. 

Nevertheless,
there are still examples where using a holonomy bundle reduction does not extend the class of processes. For a simple example we may consider a $d$-dimensional sphere $\Sp^d:=\set{x\in\RR^{d+1}}{\abs{x}=1}$ equipped with a standard round
metric, i.e.\ the metric inherited by the standard Euclidean metric of the ambient space $\RR^{d+1}.$ Then it is not hard to see that the holonomy group is the special orthogonal group $\SO(d)$, hence isotropic \levy{} processes exhaust all \levy{} processes on a sphere. 

Still, holonomy bundle reduction often extends the class of \levy{} process and to see the difference between only isotropic \levy{} processes and all possible \levy{} processes from
Definition~\ref{def:levyonman} we now analyse a couple of examples on flat Riemannian manifolds (meaning the
metric is locally isometric to the standard Euclidean one) where isotropic \levy{} processes form a strictly smaller class of processes than all the possible processes from Definition~\ref{def:levyonman}.

\subsubsection{Tori} 
Let $n\in\N$ and consider the action of a discrete
group $\Z^n$ on $\RR^n$ given by 
$$\varphi\colon \RR^n\times \Z^n \to \RR^n, \, \varphi(x,z)=x+z.$$
The action is free and proper so the Quotient manifold theorem \cite[Thm.~21.10]{leeIntro} implies that the quotient $\T^n:=\RR^n/\Z^n$ has a unique
smooth structure making it into a smooth manifold of dimension $n$, such that the quotient projection $\pi\colon \RR^n \to \T^n$ is a smooth submersion. The smooth manifold $\T^n$ is known as an $n$-dimensional torus. Furthermore, the projection $\pi$
is a local diffeomorphism and $\RR^n$ is a universal covering space of $\T^n.$ We consider a standard Euclidean metric $g$ on $\RR^n$ and since $\varphi$ acts by isometries, there is a canonical metric $\widetilde{g}$ on $\T^n$ defined by
\begin{equation}\label{inducedmetric}
	\widetilde{g}_p(u,v):=g_x((d\pi_x)^{-1}u,(d\pi_x)^{-1}v),
\end{equation}
 where $x\in \RR^n$ is any point with $\pi(x)=p\in \T^n$ and $u,v\in T_p\T^n.$

We now compute the holonomy group. Let $\widetilde{\gamma}$ be a loop based at $p\in \T^n.$ Covering space theory tells us that for any $x\in \pi^{-1}(p)$ there exists a unique curve $\gamma$ in $\RR^n$ with
$\gamma_0=x$, such that $\pi\circ\gamma=\widetilde{\gamma}.$ The endpoint of $\gamma$ is not necessarily equal to $x$, but it is always some element of the fibre $\pi^{-1}(p),$ so it can be written as $\varphi(x,z)$ for a unique
$z\in \Z^n.$ It is now clear that the parallel
transport $\widetilde{\tau}_{\widetilde{\gamma}}$ along $\widetilde{\gamma}$ of tangent vectors in $T_p\T^n$ can be expressed in term of parallel transport $\tau_\gamma$ along $\gamma$ in $\RR^n$ and it is given by
$$\widetilde{\tau}_{\widetilde{\gamma}}=d \pi_{\varphi(x,z)}\circ\tau_\gamma \circ(d\pi_x)^{-1}=d \pi_x\circ (d\varphi(\,\cdot\,,z)_x)^{-1}\circ\tau_\gamma \circ(d\pi_x)^{-1}=\mathrm{id}_{T_p\T^n}.$$ 
In the last equality we used that the parallel transport of a flat Euclidean metric, when we identify $T\RR^n$ by $\RR^n\times \RR^n,$ is given exactly by the translation $\varphi(\,\cdot\,,z).$ Hence, the holonomy
group is trivial and the holonomy bundle is diffeomorphic to $\T^n$. It also follows that each horizontal
vector field $H_i$ on the holonomy bundle is $\pi$-related (see~\eqref{eq:relatedVF} below for the definition) to the standard vector field $\partial_i$ on
$\RR^d$ for $i=1,\ldots,d.$ This implies the geodesic completeness of the torus and shows that every \levy{} process on $\T^n$ is given as $\pi(X)$, where $X$ is
an arbitrary \levy{} process on $\RR^n.$ On the other hand isotropic \levy{} processes on $\T^n$ are exactly those given as $\pi(X)$ for an isotropic
\levy{} process $X$ on $\RR^n$ and thus form a strictly smaller class. Actually, $\Z^n$ is a closed subgroup of $\RR^n$ and the action we considered is given via the group operation, so torus $\T^n$ can actually be given a structure of a Lie group
and we can alternatively proceed as in Section~\ref{levyonlie} below to recover the same results.
\subsubsection{Generalized Klein bottles} 
The second example generalizes the Klein bottle. Let $n\in\N$ and let $\RR^{n+1}=\RR^n\times \RR$. The group we consider is a semi-direct product $\Z^n\rtimes\Z$, so that the group multiplication is given by
$(a,b)\cdot(z,w):=((-1)^w a+z,b+w)$. Note that this is not a subgroup of $\RR^{n+1}$ with group operation of addition. The group action we consider is 
$$\varphi\colon (\RR^{n}\times \RR)\times (\Z^n\rtimes\Z) \to \RR^n\times \RR, \, \varphi((x,y),(z,w))=((-1)^w x+z,y+w).$$
Again, the action is free and proper so the quotient $\K^{n+1}:=(\RR^{n}\times \RR)/(\Z^n\rtimes\Z)$ has a unique
smooth structure making it into a smooth manifold of dimension $n$. We call the manifold $\K^{n+1}$ an $(n+1)$-dimensional Klein bottle. The projection $\pi\colon \RR^{n+1}\to \K^{n+1}$ is a local diffeomorphism and
$\RR^{n+1}$ is a universal covering space of $\K^{n+1}.$ As in~\eqref{inducedmetric} we can endow $\K^{n+1}$ with a flat Riemannian metric turning in into a geodesically complete manifold. 

To compute the holonomy group we proceed as above. Let $\widetilde{\gamma}$ be a loop based at $p\in \K^{n+1}$ and $\gamma$ a unique curve in $\RR^{n+1}$ starting at $(x,y)\in\pi^{-1}(p)$, such that $\pi\circ
\gamma=\widetilde{\gamma}.$ The endpoint of $\gamma$ can be written as $\varphi((x,y),(z,w))$ for a unique element $(z,w)\in \Z^n\rtimes\Z$. Again, we may compute the parallel transport as
$$\widetilde{\tau}_{\widetilde{\gamma}}=d \pi_{\varphi((x,y),(z,w))}\circ\tau_\gamma \circ(d\pi_{(x,y)})^{-1}=d \pi_{(x,y)}\circ (d\varphi(\,\cdot\,,(z,w))_{(x,y)})^{-1}\circ\tau_\gamma \circ(d\pi_{(x,y)})^{-1}.$$
We then easily see that the map $(d\varphi(\,\cdot\, ,(z,w))_{(x,y)})^{-1}\circ\tau_\gamma$ can be expressed in terms of standard coordinates on $\RR^{n+1}$ as 
$$\begin{bmatrix}
	(-1)^wI_n&0\\
	0^\top&1
\end{bmatrix},$$
where $I_n$ is an identity matrix of dimension $n$ and $0\in \RR^n$ is a zero column vector. Thus, the holonomy group is isomorphic to $\Z_2$ and all the \levy{} processes on $\K^{n+1}$ are given by $\pi(X)$, where $X$ is a \levy{}
process on $\RR^{n+1}$ with law invariant under the multiplication by the matrix
$$\begin{bmatrix}
-I_n&0\\
0^\top&1
\end{bmatrix}.$$ Again, this is less restrictive then only using isotropic \levy{} processes on $\RR^{n+1}$ which induce isotropic \levy{} processes on $\K^{n+1}$
 
\subsection{Lie groups} \label{levyonlie}

We will use Definition~\ref{def:levyonman} to construct \levy{} processes on a connected $d$-dimensional Lie group $G$ and compare them to a classical definition of (left) \levy{} processes on Lie groups defined via independent and stationary increments (see \cite{liao_levy} or \cite{huntgroups}).

 First, we need to select a suitable connection on $G$. Lie groups are parallelizable since they possess a global frame. To
 get a global frame we consider a basis $V_1,\ldots,V_d$ of a Lie algebra $\mathfrak{g}:=T_eG$ of $G$, where $e$ denotes the identity element. Given an element
 $V\in\mathfrak{g}$ we define a left-invariant vector field $V^L$ by $V^L|_g:=d(L_g)_e V$, where $d(L_g)_e$ is a differential of left multiplication $L_g\colon h\mapsto gh$. Then a global frame is given by $V^L_1,\ldots,V^L_d$. Analogously, we could define right-invariant vector fields $V_1^R,\ldots,V^R_d,$ which also form a global frame. Having a global frame
 allows us to define a connection (covariant derivative) on elements of the global frame and then extending it to all vector fields by $C^\infty(G)$-linearity in the first argument and by the product rule in the second argument. 
 
 We want to select an appropriate connection so that the parallel transport induced by any curve $\gamma$ is given
 exactly by the differential of the left multiplication $d(L_{\gamma_1\gamma_0^{-1}})_{\gamma_0}$ where the left multiplication takes the initial point $\gamma_0$
 of the the curve to its end point $\gamma_1$. Since we want that this equality holds for any curve $\gamma$, we quickly see that the only possible such connection is given via 
 $$\nabla_{V^L_i}V^L_j=0 \quad \text{for all} \quad i,j=1,\ldots,d$$
  and we call this connection a left-invariant connection. It is easily seen that this connection turns $G$ into a geodesically complete manifold. If we wanted that the parallel transport of a connection is given by the differential of right translations, we could
 use the same prescription, but the global frame would have to consist of right-invariant vector fields $V^R_1,\ldots,V^R_d$.
 
 It is obvious that the holonomy group associated with the left-invariant connection is trivial, so for any frame $u\in F(G)$ the restricted projection $\pi\colon P(u)\to G$ is actually a diffeomorphism between a holonomy bundle $P(u)$ and a Lie
 group $G$.
 For definiteness, we fix $u$ to be the frame  with $\pi(u)=e$ for which $u(e_i)=V_i$ holds for every $i=1,\ldots,d$.
 
 In order to get
 a \levy{} process $X$ on the Lie group $G$ we first construct a horizontal \levy{} process $U$ on $P(u)$ as a solution of a Marcus SDE~\eqref{eq:horlevy}, where $H_i$ are fundamental horizontal vector fields of the left-invariant connection and then
 $X=\pi(U).$ Since $\pi\colon P(u)\to G$ is a diffeomorphism, Lemma~\ref{lemma:transformMarcusSDE} shows that the process $X$ is actually a solution of a Marcus SDE 
 $$ X_t=\pi_*H_i(X_{t-})\diamond \d Y^i_t,$$
  and due to
 our particular choice of frame $u$ and the associated holonomy bundle $P(u)$, the push-forward vector fields $\pi_*H_i$ are equal to $V^L_i$ for each $i=1,\ldots,d$, so $X$ is a solution of a Marcus SDE
\begin{equation}\label{sde:levyX}
\d X_t=V^L_i(X_{t-})\diamond \d Y^i_t.
\end{equation}

Since $V^L_i$ are left-invariant vector fields and since $Y$ is a \levy{} process it seems obvious the process $X$ has independent and stationary increments, making it a (left) \levy{}
process on a Lie group $G$ as per classical definition. However, we would also like to relate some characteristics of the \levy{}
process $Y$ to those of the process $X$. In particular, we would like to relate their \levy{} measures, which govern their jumps and are therefore a measures on $\RR^d\backslash\{0\}$ and $G\backslash \{e\}$, respectively, which satisfy certain integrability conditions.

Process $X$ is a solution of~\eqref{sde:levyX} and by the definition of the Marcus SDE when $Y$ jumps at time $t$, then
$X$ also jumps at the same time and $X_t=X_{t-}\exp(V_i\Delta Y^i_t),$ where $\exp\colon \mathfrak{g}\to G$ is an exponential map of a Lie group.
Lie algebra $\mathfrak{g}$ is isomorphic to $\RR^d$ so given a \levy{} measure $\nu$ of
$Y$ we can construct a push-forward measure $\widetilde{\mu}:=\widetilde{\exp}_*\nu$ on $G$ where $\widetilde{\exp}\colon \RR^d \ni x=x^ie_i \mapsto \exp(x^iV_i)\in G$ and then the restricted measure
$\mu:=\widetilde{\mu}|_{G\backslash\{e\}}$ should be a \levy{} measure of the process $X$. We summarize these facts in the next proposition and provide the proof in Section~\ref{proofs:levyonLie}.
\begin{prop} \label{prop:sol_is_levy}
Let $X$ be a solution of a Marcus SDE~\eqref{sde:levyX}, where $Y$ is an $\RR^d$-valued \levy{} process with a \levy{} measure $\nu.$ Then $X$ is a (left) \levy{} process on $G$ with a \levy{} measure $\mu:=\widetilde{\exp}_*\nu|_{G\backslash\{e\}}$.
\end{prop}

Proposition~\ref{prop:sol_is_levy} shows that every process $X$ constructed via Definition~\ref{def:levyonman} is a classical (left) \levy{} process in $G$, but the converse need not be true. Being able to represent a \levy{} process $X$ in $G$ as a solution of a Marcus
SDE~\eqref{sde:levyX} is equivalent to an existence of a stochastic horizontal lift $U$ of $X$, such that its anti-development $Y$ is a Euclidean \levy{} process.
By Proposition~\ref{prop:sol_is_levy} we know that \levy{} measures of processes $X$
and $Y$ then need to be related by $\mu=\widetilde{\exp}_*\nu|_{G\backslash\{e\}}.$ But the exponential map need not be surjective, so there can
exist \levy{} measures on $G$, and hence (left) \levy{} processes on $G$, which cannot be represented by this construction.
Obviously, to even have a chance of finding an anti-development $Y$, we must be able to represent all the jumps of $X$ via the exponential map. 
Actually, as soon as this is the case, we are able to construct the process $Y$.

\begin{prop} \label{prop:reprLevyOnLieViaHor}
	Let $X$ be a left \levy{} process in $G$ with a \levy{} measure $\mu$. Then there exists an $\RR^d$-valued \levy{}
	process $Y$ such that $X$ satisfies Marcus SDE~\eqref{sde:levyX} if and only if $\mu(\mathrm{Im}(\exp)^c)=0.$
\end{prop}

The proof of the proposition is given in Section~\ref{proofs:levyonLie} and we can also deduce the following corollary.

\begin{cor} \label{cor:SurjectiveLIe}
	Let $G$ be a Lie group with a surjective exponential map. Then a process $X$ is a (left) \levy{} process in $G$ if and
	only if it is a solution of a Marcus SDE~\eqref{sde:levyX} for some $\RR^d$-valued \levy{} process $Y.$
\end{cor}
\begin{remark}\begin{enumerate}[(i)]
		\item A statement analogous to the \emph{if} part of Corollary~\ref{cor:SurjectiveLIe} is stated in the review article~\cite[p.~8]{ApplebaumLiaoMarkovJumps}. In fact, Proposition~\ref{prop:sol_is_levy} shows that the
		\emph{if} part holds without the assumption on the surjectivity of the exponential map.
		\item In the \emph{only if} part, for a process $X$ to satisfy Marcus
		SDE~\eqref{sde:levyX}, one needs to construct a \levy{} process $Y$ on the extended probability space such
		that~\eqref{sde:levyX} holds. It is not sufficient
		to construct a \levy{} process $Y$ on some other unrelated probability space, such that a
		different solution $\widetilde{X}$ of~\eqref{sde:levyX} driven by $Y$ has the same law as $X$.
		The construction of the
		\levy{} process $Y$ from $X$ is given in the proof of our main result, Theorem~\ref{thm:charLevyManifolds}.
	\end{enumerate}

\end{remark}

Conditions of the corollary often hold, for example if the group is connected and compact or if it is simply connected and
nilpotent. On the other hand, we trivially observe non-surjectivity of the exponential map when a Lie group is not connected,
since then the image of the (connected) Lie algebra under the (continuous) exponential map obviously cannot be the whole
(disconnected) group. A non-trivial example of a connected (and non-compact) Lie group
with a non-surjective exponential map is given by the group $\GL(2)_+$, i.e.\ the group of matrices in $\RR^2\otimes \RR^2$ with a strictly positive determinant.

\subsection{Further examples}
Definition~\ref{def:levyonman} allows us to define \levy{} processes on an arbitrary smooth manifold as soon as we specify a connection on it.
On a general smooth manifold there are (infinitely) many possible choices of a connection and in some cases, as is seen in previous sections, there is a canonical choice. 
The particular connection we use is intimately related to the Definition~\ref{def:levyonman} and leads to often non-trivial conditions~\eqref{cond:invariance}. We will now show how these conditions
restrict and may sometimes even severely limit the class of \levy{} processes on a manifold.

Recall that the holonomy group $\Hol(u)$ of the holonomy bundle $P(u)$ (associated with a frame $u\in F(M)$) is a Lie subgroup of the general linear group $\GL(d)$, where $d$ is the dimension of the
underlying manifold $M.$ The holonomy group is determined by the connection and it is natural to consider the question of whether any Lie subgroup of the general linear group $\GL(d)$ can be realized as a holonomy
group of some manifold with the connection. A partial answer has been known for over half a century and \cite{HanoOzekiLinearConnections} shows that for any connected Lie subgroup of $\GL(d)$ there exist
some manifold with a connection, such that the induced holonomy group is exactly the selected subgroup.

Hence it is possible to find a manifold with a connection, such that its holonomy group is equal to $\GL(d)_+,$ i.e.\ invertible matrices with a strictly positive determinant. In this case the \levy{} integrator 
in the definition of a horizontal \levy{} process has to satisfy conditions~\eqref{cond:invariance} and it is quickly seen that having such a large holonomy group forces the \levy{} integrator to be a
trivial process constantly equal to $0$. Hence all \levy{} processes on such a manifold have to be constant as well. This is quite an extreme example, but it shows that the choice of the connection can have a profound effect on the size of the class of \levy{} processes. 

For a less extreme example let us consider a group consisting of matrices of the form 
$$\begin{bmatrix}
	I_{d-1}&0\\
	0^\top&\beta
\end{bmatrix},$$
where $I_{d-1}$ is an identity matrix of dimension $d-1$, $0\in \RR^{d-1}$ is a zero column vector and $\beta>0.$ Such matrices for strictly positive values of parameter $\beta$ form a connected Lie
subgroup of $\GL(d)$, hence it may once again be realized as a holonomy group of some manifold with a connection. In this case conditions~\eqref{cond:invariance} force the \levy{} integrator to stay in the
$(d-1)$-dimensional hyperplane perpendicular to $e_d$. There are no further restrictions, so we may take any \levy{} integrator, which is given as an arbitrary
$(d-1)$-dimensional Euclidean \levy{} process in the first $d-1$ components and with vanishing last component. 
This essentially means that the \levy{} process on the manifold ``loses one direction'' in which it may move, but this is a geometric restriction required to keep our definition well-posed.

\subsubsection{Cylinder with a family of connections} \label{section:cylinder}
We now present an example of a 2-dimensional manifold with a 1-parameter family of connections on it. Depending on the value of the parameter, the holonomy group turns out to be a cyclic subgroup of
$\SO(2)$, which is either finite or an infinite dense subgroup isomorphic to $\Z.$ In particular, if the rank of the holonomy group is even, it contains the element $-I_2$, providing
a class of examples where we may simplify the jump part of the generator of \levy{} processes on the manifold as in Remark~\ref{remark:simplifyGenerator}(i).
Moreover, the connection introduced below is not torsion-free (for $\alpha \neq 0$) and thus cannot be the Levi-Civita connection of a Riemannian manifold.

The $2$-dimensional manifold $M=\Sp^1\times \RR$ is paralellizable with a global frame given by the vector fields $V_\theta,\, V_z$ defined as follows: under the natural inclusion
$\Sp^1\times \RR \to \RR^2\times \RR$ let $V_\theta=(-y\frac{\partial}{\partial x} + x\frac{\partial}{\partial y})\big|_M$ and $V_z =(\frac{\partial}{\partial z})\big|_M$.
We can now specify a connection (covariant derivative) via the global frame:
$$\nabla_{V_z}V_\theta = \nabla_{V_z}V_z = 0, \quad \nabla_{V_\theta}V_\theta = - \alpha V_z,\quad \nabla_{V_\theta}V_z=\alpha V_\theta,$$ where $\alpha\in\RR$ is a real parameter.
By using~\eqref{eq:geodesic} we can bound the derivatives of geodesics (in the ambient space) which shows that $M$ equipped with this connection is geodesically complete.
Furthermore, it is easy to describe the parallel transport of such a connection as we may use a global frame to identify different tangent spaces. If we are moving along $V_z$, the tangent
vectors remain unchanged, while moving along $V_\theta$ rotates the tangent vectors
by an angle expressed in radians as $\alpha$ times the directed length of the horizontal displacement.
Hence the holonomy group is generated by the rotation around the origin through angle $2\pi\alpha$. 
The holonomy group is thus isomorphic to $\Z$ if $\alpha$ is irrational and to $(\Z_q,+)$ (for some $q\in\N$) if $\alpha$ is a rational. In particular, by taking
$\alpha = \frac{1}{2},$ the holonomy group equals $\{I_2,-I_2\}.$

\section{Proofs} \label{proofsandTechnicalresults}

\subsection{Marcus SDEs on manifolds} \label{sec:MarcusSDEResults}
We start with the proof of our main existence and uniqueness result for Marcus SDEs.

\begin{proof}[Proof of Proposition~\ref{prop:MarcusWellDef}]
	In~\cite{genStratonovich} Marcus SDE was studied in Euclidean spac\-es. They use a component-wise definition of a solution; however, the chain rule~\cite[Prop.~4.2]{genStratonovich} shows that it is
	equivalent to Definition~\ref{def:MarcusSDE} when the smooth manifold $M$ is taken to be the Euclidean space and the vector fields are Lipschitz continuous, so that the solution is defined for all times in $\RR_+$. 
	To prove the proposition for a general smooth manifold $M$, we use Whitney's embedding theorem to embed $M$ as a closed submanifold of the Euclidean space $\RR^n$ for a sufficiently large\footnote{Twice the dimension of
		the manifold $M$ suffices.} $n$. In particular, this allows us to identify the point $\partial$ of compactification $\widehat{M}$ with the point $\infty$ of compactification of $\RR^d.$
	We can then smoothly extend  vector fields $V_i$ on $M$ to vector fields $V_i^*$ on the whole $\RR^n$.
	Being smooth, these vector fields are locally Lipschitz continuous. We consider closed balls $\B_k:=\set{x\in\RR^n}{\abs{x}\le k}$ for $k\in \N$ and for each $k\in\N$ we take smooth functions $\varphi_k$ such
	that $\varphi_k=1$ on $\B_k$ and $\supp\varphi_k\subseteq \B_{k+1}$. By denoting $V^{(k)}_i:=\varphi_kV^*_i$ for $i=1,\ldots,\ell$ we see that vector fields $V^{(k)}_i$ are compactly supported, hence they are globally Lipschitz continuous and any linear combination of them is a complete vector field. 
	
	We consider a Marcus SDE 
	\begin{equation} \label{eq:restrictedMarcusSDE}
		\d X^{(k)}_t= V^{(k)}_i(X^{(k)}_{t-})\diamond \d W^i_t
	\end{equation}
	on $\RR^n$ with the initial condition $X^{(k)}_0\in M$ and without loss of generality (by translation) we may assume that $X^{(k)}_0=0\in \B_k$. Then~\cite[Thm.~3.2]{genStratonovich} shows that there exists a unique
	solution $(X^{(k)}_t)_{t\in\RR_+}$ of~\eqref{eq:restrictedMarcusSDE} and since vector
	fields $V^{(k)}_i$ are tangent to $M$, \cite[Prop~4.2]{genStratonovich} shows that the solution $X^{(k)}$ stays on $M$. We define
	$$\tau_k:=\inf\set{s\in\RR_+}{\exists \theta \in [0,1] \text{ such that } \exp(\theta V_i\Delta W^i_s)(X^{(k)}_{s-})\notin \B_k},$$ i.e.\ $\tau_k$
	is the first time the process $X^{(k)}$ or the ``hidden'' trajectory of its jump exits $\B_k.$
	For $k\le \ell$ consider the processes $X^{(k)}$ and $X^{(\ell)}$. Restricted to the interval $[0,\tau_k),$ they depend on their corresponding vector fields $V_i^{(k)}$ and $V_i^{(\ell)}$
	restricted to $\B_n$, but the vector fields coincide there. 
	Therefore, uniqueness of solutions from~\cite[Thm.~3.2]{genStratonovich} means that the processes $X^{(k)}$ and $X^{(\ell)}$ coincide on $[0,\tau_k),$ so we may unambiguously define a process
	$(X_t)_{t\in[0,\tau)}$ where $\tau=\lim_{k\to\infty}\tau_k$ so that  $X_t=X^{(k)}_t$ on $[0,\tau_k).$ It easily follows that the process $X$ is exactly the unique solution of Marcus SDE~\eqref{def:solSDEgen} on
	$[0,\tau)$. Finally, we have to prove that $\tau$ is an explosion time and that $\set{\tau_n}{n\in \N}$ form its announcing sequence. This immediately follows once we show that $\tau_n < \tau$ a.s.\ for every $n\in \N.$
	
	In general, it is possible that $X^{(k)}$ exits $\B_k$ at a time $\tau_k$ by a jump -- such a jump being induced by the jump $\Delta W_{\tau_k}$ of the integrator. If the vector field $V_i\Delta W^i_{\tau_k}$ were not complete
	then $\exp(V_i\Delta W^i_{\tau_k})(X_{\tau_k-})$ might not be defined and in this case $\tau_\ell=\tau_k$ for all $\ell\ge k$, and hence
	$\tau_k=\tau.$ However, this cannot happen if as assumed all the vector fields $a^iV_i$ are complete, since then the flow exponential map is always defined, thus the norm of $\exp(V_i\Delta W^i_{\tau_k})(X_{\tau_k-})$ is
	smaller than $L$ for some $L\in\N$ larger than $k$ and we have $\tau_k<\tau_L\le\tau$ for any $k\in \N$ and this finishes the proof.
\end{proof}

We now prove a couple of technical lemmas which elucidate the transformation rules of Marcus SDEs. First, we may consider a Marcus SDE and linearly transform the vector fields appearing in the SDE.
Next lemma shows that this is the same as appropriately linearly transforming the integrator.

\begin{lemma}\label{lemma:linearchangeMarcusSDE}
	Let $\ell,m\in \N$ and let $V_1,\ldots,V_\ell\in \Gamma(TM)$ be smooth vector fields on $M$ and $W$ be an $\RR^m$-valued \cadlag{} semimartingale. Suppose that $\widetilde{V}_1,\ldots,\widetilde{V}_m \in
	\Gamma(TM)$ are smooth vector fields given as $\widetilde{V}_i:=A^j_i V_j$ for some matrix 
	$A\in \RR^{\ell}\otimes\RR^m.$ Then an $M$-valued semimartingale $X$ defined on $[0,\tau)$ is a solution of Marcus SDE
	\begin{equation*}
		\d X_t = \widetilde{V}_i(X_{t-})\diamond \d W^i_t,
	\end{equation*} on $[0,\tau)$ if and only if it is a solution of Marcus SDE 
	$$\d X_t = V_j(X_{t-})\diamond \d \widetilde{W}^j_t$$
	on $[0,\tau),$ where an $\RR^\ell$-valued semimartingale $\widetilde{W}$ is given by $\widetilde{W}:=AW.$ 
\end{lemma}
\begin{proof}
	The defining property~\eqref{marcusDefProp} for the first Marcus SDE is the equation
	\begin{equation*}
		f(X_t)-f(X_0)=\int_0^tA^j_iV_j f(X_{s-})\circ\d W^i_s + \sum_{0<s\le t} \left( f\big(\exp(A^j_iV_j\Delta W^i_s)(X_{s-})\big)-f(X_{s-})-A^j_iV_jf(X_{s-})\Delta W^i_s\right)
	\end{equation*}
	for any $f\in C^\infty(M),$ and simple algebra using $\widetilde{W}^j=A^j_iW^i$ shows that this is also exactly the defining property for the second Marcus SDE.
\end{proof}

Another important property of the Marcus SDE is that it behaves nicely under diffeomorphisms. We will first prove a slightly more general result. Let $\Phi\colon M \to N$ be a smooth map. We say that vector fields $V\in\Gamma(TM)$ and
$V'\in\Gamma(TN)$ are \emph{$\Phi$-related} if 
\begin{equation}\label{eq:relatedVF}
	d\Phi_pV_p=V'_{\Phi(p)}
\end{equation}
for every $p\in M,$ where $d\Phi_p$ represents the differential of smooth map $\Phi$ at a point $p.$ 

\begin{lemma} \label{lemma:MarcusSDEandRelatedVF}
	Let $M$ and $N$ be smooth manifolds and let  $\Phi\colon M \to N$ be a smooth map. Let $V_1,\ldots,V_\ell\in\Gamma(TM)$ be vector fields and let semimartingale $X$ defined on $[0,\tau)$ be an
	$M$-valued solution of a Marcus SDE 
	\begin{equation} \label{eq:pretransform}
		\d X_t= V_i(X_{t-})\diamond \d W^i_t
	\end{equation}
	on $[0,\tau)$.
	Suppose that $V'_1,\ldots,V'_\ell\in \Gamma(TN)$ are vector fields, such that $V_i$ and $V'_i$ are $\Phi$-related for every $i=1,\ldots,\ell.$ 
	Then the $N$-valued semimartingale $\widetilde{X}:=\Phi(X)$ is a solution of a
	Marcus SDE 
	\begin{equation} \label{eq:posttransform}
		\d \widetilde{X}_t= V'_i(\widetilde{X}_{t-})\diamond \d W^i_t
	\end{equation}
	on $[0,\tau)$.
\end{lemma}
\begin{proof}
	We need to show that the defining property~\eqref{marcusDefProp} of Marcus SDE~\eqref{eq:posttransform} holds for $\widetilde{X}$ when we take an arbitrary smooth function $f\in C^\infty(N)$. We can take a function
	$f\circ\Phi\in C^\infty(M)$ and use the defining property~\eqref{marcusDefProp} of Marcus SDE~\eqref{eq:pretransform} for $X$ to get
	\begin{align} \label{eq:substituteAllInMarcus}
		f(\widetilde{X}_t)-f(\widetilde{X}_0)&=\int_0^tV_i(f\circ\Phi)(X_{s-})\circ\d W^i_s\\
		& + \sum_{0<s\le t} \left( f\circ \Phi\big(\exp(V_i\Delta W^i_s)(X_{s-})\big)-f(\widetilde{X}_{s-})-V_i(f\circ\Phi)(X_{s-})\Delta W^i_s\right) \nonumber.
	\end{align}
	
	Equation~\eqref{eq:relatedVF} implies that $V'_i f(\Phi(p))=V_i(f\circ\Phi)(p)$ 
	holds for any $f\in C^\infty(N)$ and $p\in M,$ so we have
	\begin{equation} \label{relatedVFtransformed}
		V'_i f(\widetilde{X}_{t-})=V_i(f\circ\Phi)(X_{t-})
	\end{equation} 
	for any $i=1,\ldots,\ell$.
	
	Using linearity in \eqref{eq:relatedVF} shows that vector fields $V_x:=x^iV_i$ and $V'_x:=x^iV'_i$ are $\Phi$-related for any $x=x^ie_i\in \RR^d.$ By \eqref{eq:relatedVF}, if a tangent vector of a curve $\gamma$
	is equal to $V_x|_p$ at $p$, then the tangent vector of a curve $\Phi\circ \gamma$ is equal to $V'_x|_{\Phi(p)}$ at $\Phi(p).$  In particular,
	$\Phi$ maps integral curves of $V_x$ to integral curves of $V'_x$, hence
	\begin{equation} \label{transformdiffeoexponential}
		\Phi\big(\exp(V_x)(p)\big)=\exp(V'_x)(\Phi(p))
	\end{equation} holds for all $p\in M.$ 
	
	Inserting~\eqref{relatedVFtransformed} and~\eqref{transformdiffeoexponential} into  equation~\eqref{eq:substituteAllInMarcus} yields exactly the defining property for Marcus SDE~\eqref{eq:posttransform} and $\widetilde{X}$ is its solution.
\end{proof} 

When a map $\Phi\colon M\to N$ is a diffeomorphism, we can define a \emph{push-forward $\Phi_*\colon \Gamma(TM)\to\Gamma(TN)$} so that for a vector field $V\in\Gamma(TM)$ the \emph{push-forward vector field $\Phi_*V\in \Gamma(TN)$} is defined by the prescription
\begin{equation} \label{def:pushforwardVF}
	(\Phi_*V)_q:=d\Phi_{\Phi^{-1}(q)}V_{\Phi^{-1}(q)}
\end{equation}
for any $q\in N$. It is clear from this definition that the vector fields $V$ and $\Phi_*V$ are $\Phi$-related, so the following lemma is an immediate corollary of Lemma~\ref{lemma:MarcusSDEandRelatedVF}. 
\begin{lemma} \label{lemma:transformMarcusSDE}
	Let $\Phi\colon M \to N$ be a diffeomorphism and let semimartingale $X$ defined on $[0,\tau)$ be an $M$-valued solution of a Marcus SDE 
	\begin{equation*} 
		\d X_t= V_i(X_{t-})\diamond \d W^i_t
	\end{equation*}
	on $[0,\tau).$
	Then the $N$-valued semimartingale $\widetilde{X}:=\Phi(X)$ is a solution of a
	Marcus SDE 
	\begin{equation*} 
		\d \widetilde{X}_t= \Phi_*V_i(\widetilde{X}_{t-})\diamond \d W^i_t
	\end{equation*}
	on $[0,\tau).$
\end{lemma}

\subsection{Proofs of Proposition~\ref{prop:wellDefLevyonman}, Lemma~\ref{lemma:genhorlevy} and Proposition~\ref{prop:genX}} \label{easierproofs}

\begin{proof}[Proof of Proposition~\ref{prop:wellDefLevyonman}]
	In the discussion prior to the Definition~\ref{def:levyonman} it was already shown that the invariance of the integrator under the holonomy group implies Markovianity of the projected process.
	We now show that changing the frame and the associated holonomy bundle does not alter the class of \levy{} processes on a connected smooth manifold $M$. 
	
	Consider a different frame $v\in F(M).$ Connectedness of $M$ implies that $\pi\colon P(v)\to M$ is surjective and if $v'\in P(v)$ is another frame, then $P(v')=P(v)$ and $\Hol(v')=\Hol(v)$, hence we may
	assume that $v=ug$ for some $g\in\GL(d)$ so that $\Hol(v)=g^{-1}\Hol(u)g$ holds \cite[Prop.~II.4.1]{kobayashinomizu}. We take a \levy{} process $X=\pi(U),$ where $U$ is a
	horizontal \levy{} process and a solution of a Marcus SDE~\eqref{eq:horlevy} with a $\Hol(u)$-invariant integrator $Y$ and $U_0\in P(u).$ Hence, if we denote $\widetilde{U}:=Ug$ and $\widetilde{Y}:={g^{-1}Y}$,
	then Lemmas~\ref{lemma:transformMarcusSDE},~\ref{lemma:changeHorizontal} and~\ref{lemma:linearchangeMarcusSDE} imply that $\widetilde{U}$ is a
	$P(v)$-valued solution of the Marcus SDE
	$$\d \widetilde{U}_t=H_i(\widetilde{U}_{t-})\diamond \d \widetilde{Y}^i_t$$
	with an integrator $\widetilde{Y},$ which is $\Hol(v)$-invariant. Since $X=\pi(U)=\pi(\widetilde{U}),$ $X$
	is also a \levy{} process when defined through the holonomy bundle $P(v)$. This shows that the same processes are called \levy{} processes on $M$ no matter which holonomy bundle we use for their definition. 
\end{proof}

\begin{proof}[Proof of Lemma~\ref{lemma:genhorlevy}]
	It is a well-known fact that a Euclidean \levy{} process $Y$ with the generating triplet $(a,\nu,b)$  can be expressed via the \levy{}-It\^o decomposition \cite[Thm.~2.4.16]{AppLevyBook} 
	\begin{equation} \label{levyitodecomp}
	Y^i_t=b^i t + \sigma^i_j B^j_t + \int_0^t\int_{0<\abs{x}<1}x^i\widetilde{N}(\d s,\d x) +  \int_0^t\int_{\abs{x}\ge 1}x^i N(\d s,\d x), 
	\end{equation}
	where $\sum_{k=1}^d\sigma^i_k \sigma^j_k=a^{ij},$
	$B$ is a standard $d$-dimensional Brownian motion,  $N$ is an independent Poisson random measure on $\RR_+\times \RR^d$ with mean measure $\leb\otimes\nu$, i.e.\ $\E[N([0,t],A)]=t\nu(A)$,  and $\widetilde{N}$ is a compensated random measure defined by $\widetilde{N}([0,t],A)=N([0,t],A)-t\nu(A).$
	Since $U$ is a solution of Marcus SDE~\eqref{eq:horlevy}, we can take $f\in C^{\infty}(P(u))$ and substitute~\eqref{levyitodecomp} into the defining property of Marcus SDEs to get
	\begin{align*}
	f(U_t)&=f(U_0)+b^i\int_0^t H_i f(U_{s-})\d s +\sigma_j^i 
	\int_0^t H_i f(U_{s-})\circ \d B^j_s\\
	&+\int_0^t\int_{0<\abs{x}<1}x^iH_i f(U_{s-})\widetilde{N}(\d s,\d x) 
	+\int_0^t\int_{\abs{x}\ge 1}x^iH_i f(U_{s-})N(\d s,\d x)\\
	&+ \int_0^t\int_{\RR^d\backslash\{0\}}\left( f(\exp(H_x)(U_{s-}))-f(U_{s-})-x^iH_i f(U_{s-})\right)N(\d s,\d x).
	\end{align*}
	After recalling the definition of Stratonovich integral and some algebraic manipulation using $\widetilde{N}(\d s, \d x)=N(\d s, \d x) - \d s\, \nu(\d x)$ we can rewrite the equation as
	\begin{align}
	f(U_t)&=f(U_0) +\sigma_j^i\int_0^t H_i f(U_{s-})\d B^j_s +\int_0^t\int_{\RR^d\backslash\{0\}}\left( f(\exp(H_x)(U_{s-}))-f(U_{s-})\right)\widetilde{N}(\d s,\d x)  \label{eq:horsubstituted}\\
	&+\int_0^t\mathscr{L}_U(f)(U_{s})\d s, \nonumber
	\end{align}
	where $\mathscr{L}_U$ is given by \eqref{eq:generatorHorLevy}. Since the first two integrals in the right hand side of equation~\eqref{eq:horsubstituted} are (local) martingales, standard arguments show that
	$\mathscr{L}_U$ is the (extended) generator (see\cite{DynkinBook,CJPSMarkovSemi}) of the horizontal \levy{} process $U$, at least on smooth functions.
\end{proof}

\begin{proof}[Proof of Proposition~\ref{prop:genX}]
	Since $f(X)=f\circ \pi(U),$ the only possible prescription for the generator of $X$ is given by
	\begin{equation} \label{eq:generatorPush}
		\mathscr{L}_X(f)(p)= \mathscr{L}_U(f\circ \pi)(v),
	\end{equation}
	where $v$ is any frame in $P(u)$ with $\pi(v)=p$ and $\mathscr{L}_U$ is a generator of a horizontal \levy{} process obtained in Lemma~\ref{lemma:genhorlevy}.
	We use invariance conditions~\eqref{cond:invariance} to apply \cite[Thm.~10.13]{DynkinBook} from which it follows
	that right hand side of \eqref{eq:generatorPush} is indeed well-defined and equal to the generator of $X$. The theorem also identifies the domain of the generator.
	
	In order to finish the proof and get the explicit expression~\eqref{eq:generatorX} for the generator we
	simply use in the above expression the equality~\eqref{bothExponentials} relating flow and geodesic exponential maps.
\end{proof}

\subsection{Proof of Theorem~\ref{thm:uniqueLift}} \label{proofHorLift}

For the proof of the theorem we will need the following estimate for the flow exponential map of the fundamental horizontal vector fields in local coordinates on the frame bundle.

\begin{lemma} \label{TaylorEstimateLemma}
	Consider local coordinates on $O\subseteq M$ and the associated local coordinates on $\pi^{-1}(O)\subseteq F(M).$ Let $u=(x^i,r^k_m)=(x,r)\in\pi^{-1}(O),\,c=c^ne_n\in \RR^d$ and consider
	$\exp(H_c)(u)=\psi(1),$ where $\psi(t)=\psi(t,u,c)$ is the integral curve of $H_c$ starting at $\psi(0)=u$, which we assume stays in $\pi^{-1}(O)$ for all $t\in[0,1].$ Then we have 
	\begin{equation} \label{taylorEstimateFlow}
	\exp(H_c)(u)= (x^i,r^k_m) + \left(c^n r^i_n,-c^n r^j_n r^l_m\Gamma^k_{jl}(x)\right)+ c^n c^{n'}F_{nn'}(\psi(\theta))
	\end{equation} for some smooth functions $F_{nn'}$ with values in $\RR^{d+d^2}$ and some $\theta\in(0,1).$ 
\end{lemma}

\begin{proof}
	Via local coordinates we may consider $\psi$ as a smooth curve in $\RR^{d+d^2}$ and we can use Taylor approximation up to order 2 to write:
	$$\exp(H_c)(u)=\psi(1)=\psi(0)+\dot{\psi}(0)+\frac{1}{2}\ddot{\psi}(\theta)$$ for some $\theta\in(0,1).$
	We obviously have $\psi(0)=u=(x^i,r^k_m).$ Since $\psi$ is an integral curve of a fundamental horizontal vector field, we use the representation in local coordinates from Lemma~\ref{lemma:horInlocal} to write
	\begin{equation} \label{firstderivativePsi}
	\dot{\psi}(t)=H_c(\psi(t))=\left(c^n\psi^i_n(t),-c^n\psi^j_n(t)\psi^l_m(t)\Gamma^k_{jl}\left(\pi \circ\psi(t)\right)\right),
	\end{equation}
	from which $\dot{\psi}(0)=\left(c^n r^i_n,-c^n r^j_n r^l_m\Gamma^k_{jl}(x)\right)$ easily follows. Hence, we have already obtained the first two terms in~\eqref{taylorEstimateFlow} and we only need
	to obtain the final one. We compute further time derivative of~\eqref{firstderivativePsi} in which we can use~\eqref{firstderivativePsi} once more to get
	\begin{align*}
	\ddot{\psi}^i(t)&=c^n\dot{\psi}^i_n(t)=-c^nc^{n'} \psi^j_{n'}(t)\psi^l_n(t)\Gamma^i_{jl}\left(\pi \circ\psi(t)\right)\\
	\ddot{\psi}^k_m(t)&=c^nc^{n'}\psi^\alpha_{n'}(t)\psi^\beta_n(t)\Gamma^j_{\alpha \beta}\left(\pi \circ\psi(t)\right)\psi^l_m(t)\Gamma^k_{jl}\left(\pi \circ\psi(t)\right)\\
	&+c^nc^{n'}\psi^j_n(t)\psi^\alpha_{n'}(t)\psi^{\beta}_m(t)\Gamma^l_{\alpha\beta}\left(\pi \circ\psi(t)\right)\Gamma^k_{jl}\left(\pi \circ\psi(t)\right)\\
	&-c^nc^{n'}\psi^j_n(t)\psi^l_m(t)\partial_\alpha\Gamma^k_{jl}\left(\pi \circ\psi(t)\right)\psi^{\alpha}_{n'}(t),
	\end{align*} 
	so we can indeed write $\frac{1}{2}\ddot{\psi}(\theta)=c^nc^{n'}F_{nn'}(\psi(\theta))$ for some smooth functions $F_{nn'}$, and thus establish the results of the lemma. 
\end{proof}

\begin{proof}[Proof of Theorem~\ref{thm:uniqueLift}]	
	We will construct a horizontal lift and an anti-development using local charts and then patch the local solutions together.
	We consider local coordinates $\varphi^\alpha$ on a precompact set $O\subseteq M$ and the associated local coordinates $\phi^\alpha,\phi^k_m$ on $\pi^{-1}(O)\subseteq F(M)$. Inspired by
	the deterministic horizontal lift equation~\eqref{horliftlocal} and condition~\eqref{jumpsforlift}
	we claim that the horizontal lift $U=(\varphi^i(X),R^k_m)$ solves the following SDE written in local coordinates:
	\begin{align} \label{cons:horlift}
	{R}^k_m(t)&=R^k_m(0)-\int_0^tR^l_m(s-)\Gamma^k_{jl}(X_{s-}) \circ \d \varphi^j(X_s)\\
	&+\sum_{0<s\le t}\left(\exp\left(H_{U^{-1}_{s-}q(X_{s-})J_s}\right)(U_{s-})^k_m-R^k_m(s-)+R^l_m(s-)\Gamma^k_{jl}(X_{s-}) \Delta \varphi^j(X_s)\right),\nonumber
	\end{align} for every $t<\widetilde{\tau},$ where $\widetilde{\tau}=\inf\set{s\in \RR_+}{X_s\notin O}$.
	Its anti-development $W$ is then defined by 
	\begin{equation} \label{cons:antidevelopment}
	W^i_t=\int_0^t \overline{R}^i_j(s-)\circ\d \varphi^j(X_s) + \sum_{0<s\le t}\left((U^{-1}_{s-}q(X_{s-}) J_s)^i-\overline{R}^i_j(s-)\Delta \varphi^j(X_s) \right)
	\end{equation}
	for $t< \widetilde{\tau}$, where $\overline{R}$ is an inverse of $R$, i.e.\ $\overline{R}^i_jR^k_l=\delta^i_l\delta^k_j.$ 
	Note that such $U$ and $W$ will be adapted to $(\mathcal{G}_t)_{t\in \RR_+}.$

	We first analyse the equality~\eqref{cons:horlift}. For it to even make sense, we need to show that the sum on the right hand side is a.s.\ finite. We assume that we already have such a process $U$  which is \cadlag{},
	hence $U_-$ is locally bounded and we denote its trajectory up to time $t$ by $\kappa_t(\omega)$, which is a bounded set a.s. Hence there
	exists $\varepsilon(\omega)>0$ and an a.s.\ bounded set $D_t(\omega)\subseteq \pi^{-1}(O)$ such that every integral curve of a horizontal vector fields $H_c$ with initial point in $\kappa_t$ and $\abs{c}<\varepsilon$
	stays\footnote{This follows from results about systems of ODEs in Euclidean setting and uses only that the
	set of initial points is bounded and that the vector field is locally Lipschitz in an argument as well as in parameter.} in $D_t$ up to time $1$.
	We denote $c_s=U^{-1}_{s-}q(X_{s-})J_s,$ where $U^{-1}_{s-}q(X_{s-})$ is a locally bounded holonomy group valued process. This local boundedness together with condition~\eqref{jumpsforlift} implies that there are at most finitely many times $s\in[0,t]$
	with $\abs{c_s}\ge\varepsilon$, so the sum of absolute values of terms corresponding to such times is a finite random variable which we denote by $C_t$. Since $X_{s-}=\pi(U_{s-})$ and $X_s=\pi\left(\exp\left(H_{c_s}\right)(U_{s-})\right),$ we
	can use Lemma~\ref{TaylorEstimateLemma} and write each of the remaining terms corresponding to $\abs{c_s}<\varepsilon$ as
	\begin{align*}
	&R^k_m(s-)-c_s^n R^j_n(s-) R^l_m(s-)\Gamma^k_{jl}(X_{s-})+ c_s^nc_s^{n'}(F_{nn'})^k_m(\psi(\theta,U_{s-}))-R^k_m(s-)\\
	+&R^l_m(s-)\Gamma^k_{jl}(X_{s-}) \left(\varphi^j(X_{s-}) + c_s^n R^j_n(s-) + c_s^nc_s^{n'}(F_{nn'})^j(\psi(\theta,U_{s-}))-\varphi^j(X_{s-})\right)\\
	=&c_s^nc_s^{n'}\left(\widetilde{F}_{nn'}(\psi(\theta,U_{s-}))\right)^k_m,
	\end{align*} for some smooth functions $\widetilde{F}_{nn'}$ and some $\theta=\theta(s,\omega)\in (0,1).$ Triangle and Cauchy-Schwarz inequalities then imply that the absolute value of the whole sum is smaller or equal to
	\begin{align*}
	&C_t+\sum_{0<s\le t,\abs{c_s}<\varepsilon}  \abs{c_s^n c_s^{n'}\left(\widetilde{F}_{nn'}(\psi(\theta,U_{s-}))\right)^k_m}\\
	\le& C_t +d^{-1}\sup_{1\le n,n'\le d,\theta\in [0,1],s\in[0,t],\abs{c_s}<\varepsilon}\abs{\left(\widetilde{F}_{nn'}(\psi(\theta,U_{s-}))\right)^k_m}\sum_{0<s\le t,\abs{c_s}<\varepsilon} \abs{c_s}^2\\
	\le&C_t+ K_tL_t^2d^{-1}\sum_{s\le t}\abs{J_s}^2<\infty,
	\end{align*}
	where $L_t$ is the (local) bound for the holonomy group valued process $U^{-1}_{s-}q(X_{s-})$ and we have used that the final sum is finite by~\eqref{jumpsemiBound}.  We obtain the bound $K_t(\omega)$ for the
	supremum as a consequence of taking a supremum of finitely many continuous functions over a subset of a bounded set $D_t$.

	Next, we need to show that there exists a solution of SDE~\eqref{cons:horlift} which is a semimartingale.
	We wish to use the general theory of SDEs against semimartingales, so we need
	to rewrite the equation into a correct
	form. We start by considering SDE~\eqref{cons:horlift} in the matrix form so that the first integral can be written as 
	$$\int_0^t F(R(s-))\circ\d \varphi(X_s)=\int_0^t F_j(R(s-))\circ\d \varphi^j(X_s),$$
	where the random coefficients $F_j\colon (\RR^d\otimes\RR^d)\times \Omega\to \RR^d\otimes\RR^d$ are given by
	$$F_j(r,\omega)^k_{mj}:=r^l_m\Gamma^k_{jl}(X_{s-}(\omega)) $$
	for $j=1,\ldots,d$
	and $\varphi(X)$ is an $\RR^d$-valued semimartingale $\varphi(X):=(\varphi^1(X),\ldots,\varphi^d(X))^\top.$
	Note that as a function of $r,$ the coefficient $F$ is linear.
	Next, we need to appropriately transform the sum. 
	We view the sum as an $\RR^d\otimes\RR^d$-valued processes and we may use Lemma~\ref{lemma:changeexponential} to express the sum as a component-wise integral
	$$\int_0^tA(X_{s-},J_s)R(s-)\d Z_s,$$
	where $Z_t=\sum_{0<s\le t}\abs{J_s}^2$ is a (finite variation) real semimartingale. The $\RR^d\otimes\RR^d$-valued process $A(X_{t-},J_t)$ has components given by
	$$A(X_{t-},J_t)^k_m:=\left(\exp\left(H_{\widetilde{U}^{-1}_{t-}q(X_{t-})J_t}\right)(\widetilde{U}_{t-})^k_m-\delta^k_m+B(X_{t-},J_t)^k_m\right)/\abs{J_t}^2,$$
	where $\delta^k_m=\ind(k=m)$ is Dirac delta, the process $\widetilde{U}$ is given by $\widetilde{U}:=UR^{-1},$
	and the $\RR^d\otimes\RR^d$-valued process $B(X_{t-},J_t)$ has components given by
	$$B(X_{t-},J_t)^k_m:=\Gamma^k_{jm}(X_{t-}) \Delta \varphi^j(X_t)=\Gamma^k_{jm}(X_{t-}) \left( \varphi^j(\Exp_{X_{t-}}(q(X_{t-})J_t))-\varphi^j(X_{t-})\right).$$
	Notice that $\widetilde{U}$ has a simple representation in local coordinates as $\widetilde{U}_t=(X_t,I_d),$ so that $A(X_{t-},J_t)$ is indeed just a function of $X_{t-}$ and $J_t$ and there is no dependence on the process $R.$
	Therefore, we can rewrite SDE~\eqref{cons:horlift} in matrix form as
	\begin{equation} \label{eq:matrixhorlift}
		R(t)=R(0)+\int_0^t F(R(s-))\circ\d \varphi(X_s)+ \int_0^tA(X_{s-},J_s)R(s-)\d Z_s.
	\end{equation}
	Note that the integrands in~\eqref{eq:matrixhorlift} are linear in $R(s-)$ with random linear coefficients given by $\Gamma^k_{jl}(X_{s-})$ and $A(X_{s-},J_s)$. 
	Thus~\eqref{eq:matrixhorlift} is a standard form of an SDE against semimartingales with the integrands being linear, hence functional Lipschitz
	(in the terminology of~\cite[Ch.~V]{ProtterIntegration}). More precisely, 
	the integrands are obtained via an action of a functional Lipschitz operator on the process $R$. 
	Functional Lipschitzness is a consequence of linearity and a.s.\ boundedness of linear coefficients on finite intervals. 
	Indeed, for a.e.\ $\omega$ and every $t>0$ we have $\sup_{s\le t}\abs{\Gamma^k_{jl}(X_{s-}(\omega))}<\infty,$ which follows from the \cadlag{} property of $X$ and the smoothness of Christoffel symbols.
	Moreover, it also holds that $\sup_{s\le t}\abs{A(X_{s-}(\omega),J_s(\omega))}<\infty.$ Observe first that $A$ is non-zero only when $J$ is non-zero, which holds for at most countably many times, and further for every $\varepsilon>0$ we  have
	$\abs{J_s(\omega)}>\varepsilon$ for at most finitely times $s$. Thus, the finiteness of the supremum holds trivially at finitely many large-jump times. We may assume that all of the (possibly infinitely many)
	remaining small jumps are bounded by a constant $\varepsilon$. By the boundedness of
	$\widetilde{U}^{-1}_{s-}(\omega)q(X_{s-}(\omega))$ by some constant $K$ (see Remark~\ref{remark:defineJ}~(ii)) we can therefore consider $\varepsilon$ small enough, such that the transformed jumps
	$\widetilde{U}^{-1}_{s-}(\omega)q(X_{s-}(\omega))J_s(\omega)$, which are bounded by $\varepsilon K$, are small enough, so that we may apply Lemma~\ref{TaylorEstimateLemma} with
	$c=\widetilde{U}^{-1}_{s-}(\omega)q(X_{s-}(\omega))J_s(\omega)$ and $u=\widetilde{U}_{s-}(\omega)=(x^j,r^k_m)=(\varphi^j(X_{s-}(\omega)),\delta^k_m)$.
	Using~\eqref{taylorEstimateFlow} we compute
	\begin{align*}
		A(X_{s-},J_s)^k_m&=\left(\delta^k_m -c^j\Gamma^k_{jm}(X_{s-}(\omega))+c^nc^{n'}F_{nn'}(\psi(\theta))^k_m-\delta^k_m  \right)/\abs{J_s(\omega)}^2 \\
		&+\Gamma^k_{jm}(X_{s-}(\omega))\left( \varphi^j(X_{s-}(\omega))+ c^j+c^nc^{n'}F_{nn'}(\psi(\theta))^j- \varphi^j(X_{s-}(\omega))\right)/\abs{J_s(\omega)}^2\\
		&= c^nc^{n'} \left( F_{nn'}(\psi(\theta))^k_m + \Gamma^k_{jm}(X_{s-}(\omega))F_{nn'}(\psi(\theta))^j \right)/\abs{J_s(\omega)}^2\\
		&\le K^2 \sum_{n,n'=1}^d\left(F_{nn'}(\psi(\theta))^k_m + \Gamma^k_{jm}(X_{s-}(\omega))F_{nn'}(\psi(\theta))^j\right),
	\end{align*}
	from which boundedness of supremum of $A$ follows by using the \cadlag{} property of $X$, smoothness of maps
	$\Gamma^k_{jm}$ and $F_{nn'},$ and the fact that integral curves $\psi$ stay inside a compact set due to small sizes of jumps and the \cadlag{} property of $X$. 
	
	Thus, the coefficients of the SDE are indeed functional Lipschitz and \cite[Thm.~7 in~Ch.~V]{ProtterIntegration} implies that there exists a unique solution $R$ of the
	SDE~\eqref{eq:matrixhorlift} and the linearity of coefficients implies non-explosion of the solution,
	i.e.\ process $R$ is defined until the stopping time $\widetilde{\tau}$. 
	Actually, there is a potential issue as the coefficient $A(X_{s-},J_s)R(s-)$ is not predictable, but it is still optionally measurable w.r.t. $(\mathcal{G}_t)_{t\in \RR_+}$
	and the integral is against an increasing finite variation process so we may 
	still use a trivial extension of \cite[Thm.~7 in~Ch.~V]{ProtterIntegration} (see also remarks in the proof of \cite[Thm.~3.2]{genStratonovich}).
	
	We claim that the frame bundle valued process $U$ given in local coordinates as $U=(X,R)$ is exactly the horizontal lift of $X$.		
	It is clear that $U$ is a lift of $X$, i.e.\ it projects to $X$, but we still need to check that it is horizontal, more precisely, we need to identify its anti-development. The candidate for the
	anti-development is given by equation~\eqref{cons:antidevelopment}, initially only for $t<\widetilde{\tau}$, and similar calculation as above shows that the sum in~\eqref{cons:antidevelopment} is almost
	surely convergent so that $W$ is a semimartingale. We need to check that $W$ is an anti-development of $U$, i.e.\ we need to show that $U$ solves Marcus SDE~\eqref{def:solSDEgen} when the integrator $W$ is given
	by~\eqref{cons:antidevelopment}. It is enough to check the defining property of Marcus SDE~\eqref{def:solSDEgen} for functions which constitute a local chart on $\pi^{-1}(O)$.
	
	We first take $f=\phi^\alpha$, so that $\phi^\alpha(U_t)=\varphi^\alpha(X_t)$ and note that
	by the representation in local coordinates of fundamental horizontal vector fields from
	Lemma~\ref{lemma:horInlocal} we have $H_i\phi^\alpha(u)=r^\alpha_i$ for $u=(x^j,r^k_m)$. Therefore,
	\begin{align*}
	&\int_0^tH_i \phi^\alpha(U_{s-})\circ\d W^i_s + \sum_{0<s\le t} \left( \phi^\alpha(\exp(H_i\Delta W^i_s)(U_{s-}))-\phi^\alpha(U_{s-})-H_i \phi^\alpha(U_{s-})\Delta W^i_s\right)\\
	&=\int_0^tR^\alpha_i(s-)\overline{R}^i_j(s-)\circ \d \varphi^j(X_s)+ \sum_{0<s\le t} \left( R^\alpha_i(s-)(U^{-1}_{s-}q(X_{s-}) J_s)^i-R^\alpha_i(s-)\overline{R}^i_j(s-)\Delta \varphi^j(X_s)\right)\\
	&+\sum_{0<s\le t} \left( \varphi^\alpha(X_{s})-\varphi^\alpha(X_{s-})-R^\alpha_i(s-)\Delta W^i_s\right)\\
	&=\int_0^t1\circ \d \varphi^\alpha(X_s)+\sum_{0<s\le t} \left( R^\alpha_i(s-)\Delta W^i_s-\Delta \varphi^\alpha(X_s)+\Delta \varphi^\alpha(X_s)-R^\alpha_i(s-)\Delta W^i_s\right)\\
	&=\varphi^\alpha(X_s)-\varphi^\alpha(X_0)=\phi^\alpha(U_t)-\phi^\alpha(U_0)
	\end{align*}
	holds, where we used~\eqref{cons:antidevelopment} in the first equality and~\eqref{jumpsforanti}
	in the second inequality. This establishes the defining property of Marcus SDE~\eqref{def:solSDEgen} for $f=\phi^\alpha.$ 
	
	We now take $f=\phi^k_m$, so that $\phi^k_m(U_{s})=R^k_m(s)$ and note that Lemma~\ref{lemma:horInlocal} implies the equality $H_i\phi^k_m(u)=-r^\beta_i r^l_m\Gamma^k_{\beta l}(p)$ for $u=(p^i,r^k_m).$ Then	
	\begin{align*}
	&\int_0^tH_i \phi^k_m(U_{s-})\circ\d W^i_s + \sum_{0<s\le t} \left( \phi^k_m(\exp(H_i\Delta W^i_s)(U_{s-}))-\phi^k_m(U_{s-})-H_i \phi^k_m(U_{s-})\Delta W^i_s)\right)\\
	&=-\int_0^t R^\beta_i(s-) R^l_m(s-)\Gamma^k_{\beta l}(X_{s-})\overline{R}^i_j(s-)\circ \d \varphi^j(X_s)\\
	&+ \sum_{0<s\le t} \left(- R^\beta_i(s-) R^l_m(s-)\Gamma^k_{\beta l}(X_{s-})\Delta W_s^i+R^\beta_i(s-) R^l_m(s-)\Gamma^k_{\beta l}(X_{s-})\overline{R}^i_j(s-)\Delta \varphi^j(X_s)\right)\\
	&+\sum_{0<s\le t} \left( \phi^k_m(\exp(H_{U^{-1}_{s-}q(X_{s-})J_s})(U_{s-}))-\phi^k_m(U_{s-})+R^\beta_i(s-) R^l_m(s-)\Gamma^k_{\beta l}(X_{s-})\Delta W^i_s)\right)\\
	&=-\int_0^t R^l_m(s-)\Gamma^k_{j l}(X_{s-})\circ \d \varphi^j(X_s) \\
	&+ \sum_{0<s\le t}\left( \phi^k_m(\exp(H_{U^{-1}_{s-}q(X_{s-})J_s})(U_{s-}))-\phi^k_m(U_{s-})+R^l_m(s-)\Gamma^k_{j l}(X_{s-})\Delta \varphi^j(X_s)\right)\\
	&=R^k_m(t)-R^k_m(0)=\phi^k_m(U_t)-\phi^k_m(U_0)
	\end{align*}
	holds, where we used~\eqref{cons:antidevelopment} in the first equality and~\eqref{cons:horlift} in the penultimate equality and this establishes the defining property of Marcus SDE~\eqref{def:solSDEgen} for $f=\phi^k_m.$
	Thus, we have established that $W$ is an anti-development of $U$ (and $X$), hence $U$ is indeed a horizontal process and a horizontal lift of $X$. 
	
	We want to extend processes $U$ and $W$ to a larger time interval. 
	Note that we can find a (random) sequence of precompact charts $O_n$ such that $X_{\tau_{n-1}}\in O_n$ for $n\in \N$ where $\tau_0=0$ and $\tau_{n}=\inf\set{s>\tau_{n-1}}{X_{s}\notin O_n}$ for $n\in\N$ and such that $\tau_n\to \tau$ (the lifetime of $X$) as $n\to \infty.$
	Construction from above first defines horizontal lift and anti-development on $[0,\tau_1).$ At time $\tau_1$ we then define
	$U_{\tau_1}=\exp\left(H_{U^{-1}_{\tau_1-}q(X_{\tau_1-})J_{\tau_1}}\right)(U_{\tau_1-})$
	and  $W_{\tau_1}=W_{\tau_1-}+U^{-1}_{\tau_1-}q(X_{\tau_1-}) J_{\tau_1}$ and this makes $U$ horizontal lift and $W$ an anti-development of $X$ on $[0,\tau_1]$. Since $U_{\tau_1}\in\pi^{-1}(O_2)$ we can reuse the construction
	from above, where now the initial point is $U_{\tau_1}$ and this defines processes $U$ and $W$ also on $[\tau_1,\tau_2).$ Iteratively, this defines the horizontal lift and anti-development on
	$[0,\tau_n)$ so by sending $n\to \infty,$ the horizontal lift and anti-development are semimartingales defined until the lifetime of $X$.
	
	We still need to prove the uniqueness of horizontal lift and anti-development and hence
	independence of selection of local charts used in the construction. We will proceed similarly to \cite[Thm.~3.2]{horizontalLift}. Suppose that $U,W$
	are a horizontal lift and anti-development, respectively, induced by $X,J$ and suppose $V,\widetilde{W}$ are another such pair with $U_0=V_0.$ Since
	$\pi(U)=\pi(V)=X,$ there is a unique \cadlag{} process $g_s$ with values in $\GL(d)$ such that
	$V_s=U_sg_s=\rho(U_s,g_s),$ where $\rho\colon F(M)\times \GL(d)\to M$ is a right group action of $\GL(d)$ on the frame bundle. In particular, $g_0=e$ and the equality 
	$$\Delta \widetilde{W}_t=V^{-1}_{t-}q(X_{t-}) J_t=g_{t-}^{-1}U^{-1}_{t-}q(X_{t-}) J_t=g_{t-}^{-1}\Delta W_t$$ holds. Actually, it is not hard to see that
	\begin{equation} \label{relateAntiDevelopment}
	\widetilde{W}_t=\int_0^tg_s^{-1}\circ\d W_t
	\end{equation}
	holds. Since $W$ (resp.\ $\widetilde{W})$ is an anti-development of $U$ (resp.\ $V$), we can use the appropriate versions of \eqref{jumpsforlift} to compute
	$$U_tg_t=V_t=\exp(H_{\Delta \widetilde{W}_t})(V_{t-})= \exp(H_{g_{t-}^{-1}\Delta W_t})(U_{t-}g_{t-})=\exp(H_{\Delta W_t})(U_{t-})g_{t-}=U_tg_{t-},$$ 
	where we used Lemma~\ref{lemma:changeexponential} in the penultimate equality. It follows that $g_t=g_{t-}$ and hence the process $g_t$ is continuous.
	
	To prove uniqueness it is equivalent to show that the process $g_t$ from above is constantly equal to the identity. We consider local coordinates $u^\beta$ on $F(M)$ (these are all $\phi^\alpha$ and $\phi^k_m$ from
	above combined) and local coordinates $h^k$ on $\GL(d).$ For ease of notation we denote $V^\alpha_t=u^\alpha(V_t)$, $U^\beta_t=u^\beta(U_t)$ and $g^k_t=h^k(g_t)$. By the defining property of Marcus SDEs we have 
	\begin{align} \label{defForLocalOfLlift}
	V^\alpha_t&=V^\alpha_0+\int_0^t H_iu^\alpha(V_{s-})\circ\d \widetilde{W}^i_s+\sum_{0<s\le t} \left( u^\alpha\left(\exp(H_i\Delta \widetilde{W}^i_s)(V_{s-})\right)-V^\alpha_{s-}-H_i u^\alpha(U_{s-})\Delta \widetilde{W}^i_s\right)\\
	&=V^\alpha_0+\int_0^t H_i(u^\alpha\circ R_{g_s})(U_{s-})\circ \d W^i_s+\sum_{0<s\le t} \left(V^\alpha_{s}-V^\alpha_{s-}-H_i (u^\alpha\circ R_{g_s})(U_{s-})\Delta W^i_s\right), \nonumber
	\end{align}
	where we have used~\eqref{relateAntiDevelopment} and Lemma~\ref{lemma:changeiterated}. We also define $\rho^\alpha=u^\alpha\circ\rho$ so that we can alternatively write $V_t^\alpha=\rho^\alpha(U_t,g_t)$ and we use It\^o's formula to compute
	\begin{align*}
	V^\alpha_t&=V^\alpha_0+\int_0^t\frac{\partial\rho^\alpha}{\partial u^\beta}(U_{s-},g_s)\circ\d U^\beta_s+\int_0^t\frac{\partial \phi^\alpha}{\partial h^k}(U_{s-},g_s)\circ\d g^k_s\\
	&+\sum_{0<s\le t} \left( \rho^\alpha(U_s,g_s)-\rho^\alpha(U_{s-},g_s)-\frac{\partial\rho^\alpha}{\partial u^\beta}(U_{s-},g_s)\Delta U^\beta_s\right) \\
	&=V^\alpha_0+\int_0^t\frac{\partial\rho^\alpha}{\partial u^\beta}(U_{s-},g_s)H_iu^\beta(U_{s-})\circ W^i_s + \sum_{0<s\le t} \frac{\partial\rho^\alpha}{\partial u^\beta}(U_{s-},g_s) \left( U^\beta_s-U^\beta_{s-}-H_i u^\beta(U_{s-})\Delta W^i_s\right) \\
	&+ \int_0^t\frac{\partial \phi^\alpha}{\partial h^k}(U_{s-},g_s)\circ\d g^k_s +\sum_{0<s\le t} \left( V^\alpha_s-V^\alpha_{s-}-\frac{\partial\rho^\alpha}{\partial u^\beta}(U_{s-},g_s)\Delta U^\beta_s\right),
	\end{align*} where we used that $g_t$ is continuous and an analogous version of~\eqref{defForLocalOfLlift} for $U$. 
	Since $$\frac{\partial\rho^\alpha}{\partial u^\beta}(U_{s-},g_s)H_iu^\beta(U_{s-})=H_i(u^\alpha\circ R_{g_s})(U_{s-})$$
	holds by the chain rule, we see by comparing the above expressions for $V^\alpha_t$ that 
	$$\int_0^t\frac{\partial \phi^\alpha}{\partial h^k}(U_{s-},g_s)\circ\d g^k_s=0$$
	holds for all $\alpha.$ Since the matrix $[\frac{\partial \phi^\alpha}{\partial h^k}]^\alpha_k$ is of full rank (this is true since $g\mapsto \rho(u,g)$ is an embedding), we can conclude that 
	$$\int_0^t1\circ\d g^k_s=0$$ holds for all $k$. Therefore, the process $g_t$ is constant and hence constantly equal to identity. This proves that the horizontal lift and anti-development are unique and this completes the proof of the theorem.
\end{proof}

\subsection{Proof of Theorem~\ref{thm:charLevyManifolds}}
\label{proofofchar}

We already know from Theorem~\ref{thm:uniqueLift} that the anti-development (when it exists) is uniquely determined on $[0,\tau)$ by the pair of processes $X$ and $J$. Due to a particular jump structure of $X$ inferred from its generator, we will see that there will
always exist a process $J$ satisfying Definition~\ref{def:SemiJump}, but we need to find an appropriate one such that the anti-development will be a Euclidean \levy{} process. 

Hence, for the proof of Theorem~\ref{thm:charLevyManifolds} we need to construct a suitable jump process $J$ (on the extended probability space) and show that the associated anti-development is a \levy{} process with
given characteristics. Since the anti-development is a priori only defined on $[0,\tau)$, we also need to show that it can be extended past $\tau$.   

We will first focus on the construction of the appropriate jump process $J$, or more precisely, on the construction of the associated random measure on $\RR^d\times \RR_+$. This simple random measure, which we
will also denote by $J$ has (random) unit atoms at $(s,J_s)$ when $J_s\neq0.$ Before constructing it, we want to characterize it. The following lemma will be useful.
\begin{lemma} \label{lemma:transPoissMeasure}
	Let $G$ be a subgroup of $\GL(d)$ and let $N$ be a Poisson random measure on $\RR_+\times\RR^d$ with a mean measure $\leb \otimes\nu,$ where $\nu$ is a $G$-invariant \levy{} measure. Further, let $(g_t)_{t\in\RR_+}$
	be a predictable process with values in $G$. Denote by $\Phi$ a random transformation of $\RR_+\times\RR^d$ given by $$\Phi(\omega,t,z)=(t,g_t(\omega)z). $$
	Then the transformed (push-forward) measure $\Phi_*N:=N\Phi^{-1}$ is also a Poisson random measure on $\RR_+\times\RR^d$ with a mean measure $\leb \otimes\nu.$
\end{lemma}
\begin{proof}
	The compensator measure of $N$ is its mean measure $\mu:=\leb\otimes\nu$. It is enough to show that $\Phi_*\mu=\mu$, as we can then use \cite[Thm.~2.1]{KallenbergMarked1990} to deduce the claim.
	Let us take an arbitrary $I\in \mathfrak{B}(\RR_+), \, A\in \mathfrak{B}(\RR^d)$ and then
	\begin{align*}
	\Phi_*\mu(\omega)(I\times A)=\mu(\Phi^{-1}(\omega)(I	\times A))=\int_I\nu(g_t(\omega)A)\d t=\int_I\nu(A)\d t=\leb(I)\nu(A)
	\end{align*}
	holds for almost every $\omega$, where we used Fubini-Tonelli theorem in the second equality and a $G$-invariance of the measure $\nu$ in the third equality. Since sets $I\times A$ generate
	$\mathfrak{B}(\RR_+\times \RR^d)$, we have shown that $\Phi_*\mu$ is non-random and equal to $\mu$, so the claim of the lemma follows.
\end{proof}

Therefore, if $Y$ is a \levy{} process, with jumps given by a Poisson random measure $N$, which is an anti-development of $X$, then we know that $J$ needs to be given as a transformation of $N$ where each atom
at $(t,z)$ with $t<\tau$ is transformed to an atom at $(t,q^{-1}(X_{t-})U_{t-}z),$ where $U$ is solution of
Marcus SDE~\eqref{eq:horlevy} and $q$ is a section of the holonomy bundle as in Section~\ref{generalHorLift}. We do not transform the atoms at $(t,z)$ with $t\ge\tau$.  But
$q^{-1}(X_{-})U_{-}$ is a predictable $\Hol(u)$-valued process (we trivially extend it past $\tau$) and \levy{} measure $\nu$ of $Y$ is $\Hol(u)$-invariant, hence Lemma~\ref{lemma:transPoissMeasure} implies that $J$ is also a Poisson random
measure with mean measure $\leb\otimes\nu.$ Note that the condition equivalent to Definition~\ref{def:SemiJump} which relates the random measure $J$ and a process $X$ is 
\begin{equation} \label{XjumpsAndJ}
J(\{(s,z)\})=1 \implies X_s=\Exp_{X_{s-}}(q(X_{s-})z),\quad J\left(\{s\}\times \RR^d\right)=0\implies X_s=X_{s-} \quad \text{for } s<\tau.
\end{equation} 

Hence the goal is to find suitable $J$ which is a Poisson random measure and we have the following result.

\begin{prop} \label{prop:constructPoissonRandomMeasure}
	Let $X$ be a Markov process on $M$ adapted to $(\mathcal{H}_t)_{t\in \RR_+}$ with the generator $\mathscr{L}_X$  given by~\eqref{eq:generatorX}. Then on the extended probability space there exists
	a Poisson random measure $J$ on $\RR_+\times\RR^d$ with a mean measure $\leb\otimes\nu$ and adapted to an extended filtration $(\mathcal{G}_t)_{t\in \RR_+}$ such that~\eqref{XjumpsAndJ} holds.
\end{prop}
Before proceeding with the proof, we discuss the phenomena that may occur when constructing $J$. These phenomena illustrate why typically the process $X$ alone does not possess 
sufficient information and randomness (even on $[0,\tau)$) to carry out the construction.
\begin{enumerate}[(i)]
	\item Let $X_s =y\neq X_{s-}=x$, such that $B_{xy}:=\set{z\in \RR^d\backslash\{0\}}{\Exp_{x}(q(x)z)=y}$ contains exactly one element $z'$. Then it is clear that we need to put $J(\{(s,z')\})=1.$  
	\item Again let $X_s =y\neq X_{s-}=x$, but now the set $B_{xy}$ contains more than one element\footnote{These sets are almost surely non-empty due to the form of a generator of $X$.}. Of course, it is clear that we should have
	$J(\{s\}\times B_{xy})=1,$ but it is not clear which element should be chosen as the atom. If we simply fixed one possible element of $B_{xy}$ for every pair $(x,y)$, this would correctly
	prescribe the jumps of $X$ and satisfy~\eqref{XjumpsAndJ}, but $J$ might not have the correct
	intensity measure $\leb\otimes\nu$. Therefore, in this case we need some additional randomness to appropriately randomly assign the atom on the set $B_{xy}.$
	\item Finally, it is possible that $X_s =X_{s-}=x$ and there can still be jumps (atoms) in $J$ (see also the discussion following Definition~\ref{def:horprocess}). This can happen since the set $B_x:=\set{z\in\RR^d\backslash\{0\}}{\Exp_{x}(q(x)z)=x}$ may be non-empty. Thus, even if $J(\{s\}\times B_x)=1$, 
	actually no jump of $X$ occurs at time $s$. Naively, we might simply ignore these jumps, but then again the intensity measure of $J$ might not be equal to $\leb\otimes\nu.$ Therefore, if we wish to get a correct Poisson
	random measure, we need to add some of these jumps (atoms) as well, even though they will not manifest themself as jumps of $X$. They will have to be taken from an additional (independent) Poisson random
	measure on the extended probability space which we will also have to use to define $J$ after time $\tau$.
\end{enumerate}

This illustrates that for a general process $X$ we will need to extend our probability space (and filtration) to allow for this additional randomness. This is done in~\cite[Section 4]{CJrepr}, albeit they only consider a
measure $\nu$ on a real line, but essentially every step generalizes to our case of a multidimensional measure. The only step which needs some adjustment due to a slightly different setup is the following.
For ease of notation, put $k(x,z)=\Exp_{x}(q(x)z)$ and note that we could have written the jumps in the integral part of generator $\mathscr{L}_X$ in~\eqref{eq:generatorX} via the kernel $K\colon M \times \mathcal{B}(M)\to \RR_+$ defined by
$$K(x,A)=\int_{\RR^d}\ind(k(x,z)\in A)\nu(\d z) \text{ for } A\in\mathcal{B}(M\backslash\{x\}),$$
and this means that the equivalent of \cite[Lemma~3.4]{CJrepr} holds. Furthermore, the following equivalent of \cite[Lemma~4.32]{CJrepr} holds.

\begin{lemma} \label{lemma:kernelDisintegration}
	Let $K$ be a kernel as above. Then there exists a measurable function $\widetilde{k}\colon M\times M\times(0,1]\to \RR^d,$ such that 
	$$\int_0^1\int_A\ind(\widetilde{k}(x,y,u)\in B) K(x,\d y)\d u=\int_B\ind(k(x,z)\in A)\nu(\d z)$$ for every $A\in\mathcal{B}(M),B\in \mathcal{B}(\RR^d)$ and $x\in M$. Moreover, for every $x\in M$ and $K(x,\d y)\d u$ almost every $(y,u)$ we have
	$$k(x,\widetilde{k}(x,y,u))=y. $$ 
\end{lemma}

\begin{proof}[Proof of Proposition~\ref{prop:constructPoissonRandomMeasure}]
	Lemma~\ref{lemma:kernelDisintegration} allows us to deduce as in \cite[\S 4c, Eq.~(4.41) and Prop.~4.42]{CJrepr} that we can construct a random measure $\overline{J}$ on the extended probability space, such that~\eqref{XjumpsAndJ} holds (with
	$\overline{J}$ taking the role of $J$) and such that the compensator measure of $\overline{J}$ is given by
	$ \ind(k(X_{t-},z)\neq X_{t-},\ t<\tau)\d t\nu(dz)$.
	Note that the indicator appears, since only observable jumps of $X$ are used during the construction of the random measure.
	Measure $\overline{J}$ is almost the required random measure. To add the missing part, consider a further extension of the probability space, which supports an independent Poisson random measure $N$ with an intensity measure $\leb\otimes\nu$, and define a random measure
	$$J(\d t, \d z)=\overline{J}(\d t, \d z)+ \left(\ind\big(k(X_{t-},z)=X_{t-},\ \overline{J}(\{t\}\times\RR^d)=0,\ t< \tau\big)+ \ind(t\ge \tau)\right)N(\d t, \d z).$$
	It is easy to see that its compensator measure is a non-random measure $\leb\otimes\nu$, which means that $J$ is a Poisson random
	measure with mean measure $\leb\otimes\nu.$ Furthermore, since added jumps (atoms) from the measure $N$ do not manifest themself as jumps of $X$, the condition~\eqref{XjumpsAndJ} still holds and this finishes the proof of proposition.
\end{proof}
\begin{proof}[Proof of Theorem~\ref{thm:charLevyManifolds}]
	Proposition~\ref{prop:constructPoissonRandomMeasure} allows us to construct a Poisson random measure $J$ and hence the associated jump process $J$ which represents jumps of $X$. Therefore, by
	Theorem~\ref{thm:uniqueLift} there exists a uniquely determined anti-development on $[0,\tau)$ induced by
	$X$ and $J$ and we denote it by $Y.$ The only thing left to do is to show that $Y$ can be extended past $\tau$ so that it is a \levy{} process with correct characteristics. 
	
	First, note that the jump measure of $Y$ is given as a transformation of $J$ where each atom at $(t,z)$ is transformed to an atom at $(t,U^{-1}_{t-}q(X_{t-})z),$ and again Lemma~\ref{lemma:transPoissMeasure}
	implies that the jump measure of $Y$ is exactly a Poisson random measure with mean measure $\leb\otimes\nu$ restricted to times smaller than $\tau.$ 
	
	To determine the remaining characteristics, we recall that $Y$ is a semimartingale on a stochastic interval $[0,\tau)$.
	We wish to use the notion of semimartingale characteristics which are classically defined for semimartingales
	with infinite lifetime. However, as the Appendix in \cite{SchnurrSemimartingaleKilling} shows, essentially every result about semimartingales can be extended to semimartingales defined up to a predictable stopping time.
	
	Recall that every semimartingale $Y$ has a representation $Y=Y_0 + M + A,$ where $M$ is a local martingale started at $0$ and $A$ is a process of locally finite
	variation, but this decomposition is not unique. To get uniqueness we define a process
	$Y^e_t=\sum_{s\le t}\Delta Y_s \ind(\abs{\Delta Y_s}\ge 1)$ and then a process $Y-Y^e$ with bounded jumps is a special semimartingale, meaning that it has a
	unique representation $Y-Y^e= M + A$, where $A$ is a predictable finite variation process and $M$ is a
	local martingale. Then the local characteristics $(A,C,\widetilde{\Gamma}$) of a semimartingale $Y$ as in \cite[Chapter~II]{JacodShiryaevLimit} are given by:
	\begin{enumerate}[(i)]
		\item $A=(A^j)_{j=1,\ldots,n}$ is a unique predictable finite variation process in the decomposition of a special semimartingale $Y-Y^e$,
		\item $C=(C^{ij})_{i,j=1,\ldots,n}$ is the quadratic variation of the continuous local martingale part $Y^c$ of $Y$ (or equivalently of $M$) and is given by $C^{ij}=[Y^{i,c},Y^{j,c}],$
		\item $\widetilde{\Gamma}$ is a predictable compensator of the jump measure  $\Gamma(\d t,\d y)=\sum_{s\le t}\ind(\abs{\Delta Y_s}>0)\delta_{(s,\Delta Y_s)}(\d t, \d y)$ associated to the process $Y$.
	\end{enumerate}
	
	We have already established, that a jump measure of $Y$ is a Poisson random measure, hence the third characteristic is equal to a non-random measure $\leb\otimes\nu.$
	
	The semimartingale $Y$ can be written explicitly via its characteristics as in \cite[Thm.~II.2.34]{JacodShiryaevLimit} and plugging this decomposition into the defining
	property~\eqref{marcusDefProp} of Marcus SDE~\eqref{eq:horlevy} for a function $f\circ\pi,$ where $f\in C^\infty(M),$ we see that
	\begin{align*}
	f(X_t)-f(X_0)&-\int_0^t\left(H_i (f\circ \pi)(U_{s-})\d A^i+\frac{1}{2}H_iH_j(f\circ \pi)(U_{s-})\d C^{ij}_s\right)\\
	&-\int_0^t\int_{\RR^d\backslash\{0\}}\left( f(\Exp_{X_{s-}}(U_{s-}x))-f(X_{s-})-\ind(\abs{x}<1)x^iH_i (f\circ \pi)(U_{s-})\right) \widetilde{\Gamma}(\d s,\d x)
	\end{align*}
	is a local martingale.
	We can compare the above expression to the generator $\mathscr{L}_X$ from \eqref{eq:generatorX} and we get
	\begin{equation} \label{eq:charateristics}
		A^j_t=b^jt, \quad C^{ij}_t=a^{ij}t, \quad \widetilde{\Gamma}(\d t, \d z)=\d t \nu(\d z),
	\end{equation}
	which are the characteristics of $Y$ for $t<\tau.$ Note that by observing the generator we can only uniquely determine
	$A$ and $C$, and there can be more than one measure $\widetilde{\Gamma}$ for which the above equality holds (this can happen since functions $z\mapsto k(x,z)$ from above are not necessarily injective).
	However, we have already identified the third characteristic
	without directly referring to the generator. 
	
	By \cite[Thm.~II.4.19]{JacodShiryaevLimit} we know that (``constant'') characteristics as in~\eqref{eq:charateristics} imply that the semimartingale is actually a \levy{} process.
	However, $Y$ is only defined for times in $[0,\tau)$ and we need to extend it to all times in $\RR_+.$ If we
	can find a \levy{} process $\widetilde{Y}=(\widetilde{Y}_t)_{t\in\RR_+}$ with characteristics~\eqref{eq:charateristics} such that $\widetilde{Y}_t=Y_t$ for $t\in[0,\tau)$, then
	$\widetilde{Y}$ is also an anti-development of $X$ and it is exactly the \levy{} process we needed to construct to prove the theorem. 
	Existence of such a process $\widetilde{Y}$ is guaranteed by the following lemma and this concludes the proof of the theorem.
\end{proof}

\begin{lemma}
	Let $Y$ be a semimartingale on $[0,\tau)$ with characteristics~\eqref{eq:charateristics} for $t<\tau$, where $\tau$ is a predictable stopping time. Then there exists a \levy{} process
	$\widetilde{Y}=(\widetilde{Y}_t)_{t\in{\RR_+}}$ with characteristics~\eqref{eq:charateristics} such that $\widetilde{Y}_t=Y_t$ for $t\in [0,\tau).$
\end{lemma}
\begin{proof}
	Let $\set{\tau_n}{n\in \N}$ be an announcing sequence of $\tau.$ Consider an additional \levy{} process $Z$ with characteristics~\eqref{eq:charateristics} which is independent of $Y$.
	For every $n\in \N$ we may define a process $Z^{(n)}=(Z^{(n)}_t)_{t\in\RR_+}$ by
	$$Z^{(n)}_t:= Y^{\tau_n}_{t} + \ind(t\ge \tau_n)(Z_t-Z_{\tau_n})$$ for $t\in\RR_+$. Since characteristics of $Z^{(n)}$
	are also given by~\eqref{eq:charateristics}, it is a \levy{} process and obviously $Z^{(n)}_t=Y_t$ holds for $t\in[0,\tau_n]$. 
	\levy{} process $Z^{(n)}$ has the \levy-It\^o decomposition
	$$Z^{(n),i}_t=b^i t + \sigma^i_j B^{(n),j}_t + \int_0^t\int_{0<\abs{x}<1}x^i\widetilde{N^{(n)}}(\d s,\d x) +  \int_0^t\int_{\abs{x}\ge 1}x^i N^{(n)}(\d s,\d x) $$
	for $i=1,
	\ldots, d, $ where $B^{(n)}$ is a standard Brownian motion and $N^{(n)}$ a Poisson random measure with mean
	measure $\leb \otimes \nu$. Note that $N^{(n)}$ is equal to the jump measure $\Gamma$ of $Y$ on $[0,\tau_n].$
	Furthermore, for $m<n$ we deduce from
	$Z^{(n)}_t=Z^{(m)}_t$ for $t\in[0,\tau_m]$ that also $B^{(n)}_t=B^{(m)}_t$ for $t\in[0,\tau_m]$, so that we
	may unambiguously define a continuous process $B=(B_t)_{t\in[0,\tau)}$ by $B_t=B^{(n)}_t$ for $t\in[0,\tau_n].$
	
	Let us assume that there exists 
	a Brownian motion $\widetilde{B}=(\widetilde{B}_t)_{t\in\RR_+}$ such that $\widetilde{B}_t=B_t$ for  $t\in[0,\tau).$ We may then
	consider a Poisson random measure $N'$ with mean measure $\leb \otimes \nu$ which is independent of $Y$ and	$\widetilde{B}$ and define a random measure
	$$N(\d t, \d x) := \ind(t<\tau)\Gamma(\d t, \d x) + \ind( t\ge \tau)N'(\d t, \d x),$$ which is also a
	Poisson random measure with mean measure $\leb \otimes \nu$ since $\leb \otimes \nu$ is its compensator measure.
	Hence we may define a process $\widetilde{Y}=(\widetilde{Y}_t)_{t\in\RR_+}$ by 
	$$\widetilde{Y}^i_t:=b^i t + \sigma^i_j \widetilde{B}^j_t + \int_0^t\int_{0<\abs{x}<1}x^i\widetilde{N}(\d s,\d x) +  \int_0^t\int_{\abs{x}\ge 1}x^i N(\d s,\d x) $$
	for $i=1,\ldots, d.$ It is clear that $\widetilde{Y}$ is a \levy{} process with
	characteristics~\eqref{eq:charateristics} and since $\widetilde{B_t}=B_t$ for $t\in [0,\tau)$ and
	$N=\Gamma$ on $[0,\tau)$ we also have $\widetilde{Y}_t=Y_t$ for $t\in[0,\tau),$ so $\widetilde{Y}$ is the required process.
	
	Therefore, in order to finish the proof we only have to show the existence of Brownian motion $\widetilde{B}$
	as above. Note that since $B_t=B^{(n)}_t$ for $t\in[0,\tau_n]$, the process $B$ is a continuous local martingale on
	$[0,\tau)$ (see the definition in \cite[Section~3]{GetoorSharpeConformalMartingales} or in
	\cite[Exercise~(1.48)~in~Ch.~IV]{revuzyor}) and has a quadratic covariation $\langle B^i,B^j\rangle_t=\delta^{ij}t$ on $[0,\tau)$ for $i,j=1,\ldots,d.$ 
	For $m\in \N$ consider a constant stopping time $T_m \equiv m$ and a stopped process $B^{T_m},$ which is also a continuous local martingale on $[0,\tau)$.
	We have a trivial bound $\langle B^{T_m,i}\rangle_\tau \le m$ for $i=1,\ldots,d$, so we can conclude
	from the Burkholder-Davis-Gundy inequality \cite[Thm.~(4.1)~in~Ch.~IV]{revuzyor} that $\E[\sup_{t<\tau}|B^{T_m}_t|]<\infty$ and it then follows that the limit $\lim_{t\to\tau}B^{T_m}_t$ exists
	(see the discussion after \cite[Def.~3.1]{GetoorSharpeConformalMartingales} and
	\cite[Exercise~(1.48)~in~Ch.~IV]{revuzyor}). On $\{\tau<\infty\}$ we have that $T_m\ge \tau$ from some sufficiently large $m$ onward, so also the limit $\lim_{t\to\tau}B_t$ exists on the same event.
	As we may define $B_\tau := \lim_{t\to\tau}B_t$ on $\{\tau<\infty\}$ it becomes
	meaningful to consider a continuous stopped process $B^{\tau}$ and we can define a continuous process $(\widetilde{B}_t)_{t\in\RR_+}$ by
	$$\widetilde{B}_t = B^{\tau}_t + \ind(t\ge \tau)(B'_t - B'_\tau),$$ where $(B'_t)_{t\in\RR_+}$ is another
	Brownian motion independent of $B$. It is easily shown by \levy{}'s characterization that $\widetilde{B}$
	is also a Brownian motion and obviously $\widetilde{B}_t=B_t$ holds for $t\in[0,\tau),$ so it is the required process needed to finish the proof of the lemma.	
\end{proof}

\subsection{Proofs of results in Section~\ref{levyonlie}} \label{proofs:levyonLie}
\begin{proof}[Proof of Proposition~\ref{prop:sol_is_levy}]
	We recall a \levy{}-It\^o decomposition of a \levy{} process $Y$ as in~\eqref{levyitodecomp}.
	Since $X$ solves Marcus SDE~\eqref{sde:levyX} we can take $f\in C^{\infty}(G)$ and substitute~\eqref{levyitodecomp} into the defining property of the Marcus SDE to get
	\begin{align*}
	f(X_t)&=f(X_0)+b^i\int_0^t V^L_i f(X_{s-})\d s +\sigma_j^i 
	\int_0^t V^L_i f(X_{s-})\circ \d B^j_s\\
	&+\int_0^t\int_{0<\abs{x}<1}x^iV^L_i f(X_{s-})\widetilde{N}(\d s,\d x) 
	+\int_0^t\int_{\abs{x}\ge 1}x^iV^L_i f(X_{s-})N(\d s,\d x)\\
	&+ \int_0^t\int_{\RR^d\backslash\{0\}}\left( f(X_{s-}\widetilde{\exp}(x))-f(X_{s-})-x^iV^L_i f(X_{s-})\right)N(\d s,\d x).
	\end{align*}
	After splitting the domain in the last integral and some algebraic manipulation using $\widetilde{N}(\d s, \d x)=N(\d s, \d x) - \d s\, \nu(\d x)$ we get
	\begin{align}
	f(X_t)&=f(X_0)+b^i\int_0^t V^L_i f(X_{s-})\d s +\sigma_j^i\int_0^t V^L_i f(X_{s-})\circ \d B^j_s \label{eq:substituted}\\
	&+\int_0^t\int_{0<\abs{x}<1}\left( f(X_{s-}\widetilde{\exp}(x))-f(X_{s-})\right)\widetilde{N}(\d s,\d x) \nonumber \\
	&+\int_0^t\int_{\abs{x}\ge 1}\left( f(X_{s-}\widetilde{\exp}(x))-f(X_{s-})\right)N(\d s,\d x) \nonumber \\
	&+ \int_0^t\int_{0<\abs{x}<1}\left( f(X_{s-}\widetilde{\exp}(x))-f(X_{s-})-x^iV^L_i f(X_{s-})\right)\d s\, \nu(\d x) \nonumber.
	\end{align}
	
	We recall that $\widetilde{\exp}$ is a diffeomorphism from some neighbourhood $O$ of $0\in \RR^d$ to some neighbourhood $U$ of the  unit $e\in G$ and we may assume that $O\subseteq\set{x\in\RR^d}{\abs{x}<1}.$ We take a smaller open set $\widetilde{O}$ with $0\in O$ and $\mathrm{cl}(\widetilde{O})\subseteq O$ and let
	$\widetilde{U}:=\widetilde{\exp}(\widetilde{O}).$ Then we can find $l\in C^{\infty}(\RR^d,\RR^d)$ such that $l(x)=x$ on $\mathrm{cl}(\widetilde{O})$ and $\supp l \subseteq O.$ Given such a function $l$, we define $\xi\in C^{\infty}(G,\RR^d)$ by
	$$ \xi(h)=\begin{cases}
	0,& h \in \widetilde{\exp}(\supp l)^c\\
	l(\widetilde{\exp}^{-1}(h)),& h \in U.
	\end{cases}$$ Function $\xi$ is smooth since it is smooth on both open domains and agrees on their intersection. We can check that $l(x)=\xi(\widetilde{\exp}(x))$ holds. Moreover, simple calculations show that $\xi(e)=0$ and $V^L_j\xi^i(e)=\delta_j^i$.
	This means that functions $\xi^i$ are suitable coordinate functions on the neighbourhood of $e$ and since
	$\int_G\sum_{i=1}^d (\xi^i(h))^2\mu(\d h)=\int_{\RR^d}\sum_{i=1}^d(l^i(x))^2\nu(\d x)<\infty$, $\mu(\{e\})=0$ and $\mu(A^c)=\nu(\widetilde{\exp}^{-1}(A^c))<\infty$ for any open $A\ni e,$ the measure $\mu$ is a indeed \levy{} measure on $G$ (see \cite[Ch.~1]{liao_levy} for definitions).
	
	Then the last integral in~\eqref{eq:substituted} can be expressed as
	\begin{align}
	\int_0^t\int_{0<\abs{x}<1}\left( f(X_{s-}\widetilde{\exp}(x))-f(X_{s-})-\xi^i(\widetilde{\exp}(x))V^L_i f(X_{s-})\right)\d s\, \nu(\d x)
	+ c^i \int_0^t V^L_i f(X_{s-})\d s\label{eq:sub_l}
	\end{align}
	where $c^i=\int_{0<\abs{x}<1}(l^i(x)-x^i)\nu(\d x)$ for $i=1,\ldots,d$ are finite since $l(x)=x$ near $0$. We then plug~\eqref{eq:sub_l} back into~\eqref{eq:substituted}, use the fact that $l(x)=\xi(\widetilde{\exp}(x))=0$
	for $\abs{x}>1$ and manipulate the integral in the third line with $\widetilde{N}(\d s, \d x)=N(\d s, \d x) - \d s\, \nu(\d x)$ to get
	\begin{align}
	f(X_t)&=f(X_0)+\widetilde{b}^i\int_0^t V^L_i f(X_{s-})\d s +\sigma_j^i\int_0^t V^L_i f(X_{s-})\circ \d B^j_s \label{eq:plug_back_in} \\
	&+\int_0^t\int_{\RR^d\backslash\{0\}}\left( f(X_{s-}\widetilde{\exp}(x))-f(X_{s-})\right)\widetilde{N}(\d s,\d x) \nonumber \\
	&+ \int_0^t\int_{\RR^d\backslash\{0\}}\left( f(X_{s-}\widetilde{\exp}(x))-f(X_{s-})-\xi^i(\widetilde{\exp}(x))V^L_i f(X_{s-})\right)\d s\, \nu(\d x) \nonumber
	\end{align}
	where $\widetilde{b}^i=b^i+c^i.$ Additionally, we define a Poisson random measure $N'$ on $\RR_+\times (G\backslash\{e\})$ by $N'([0,t],A')=N([0,t],\widetilde{\exp}^{-1}(A'))$, where $A'\subseteq G\backslash\{e\}$ is
	any measurable set. This means that mean measure of $N'$ is $\leb\otimes\mu$ and we also define a
	compensator $\widetilde{N}'([0,t],A')=N'([0,t],A')-t\mu(A).$ We then use a change-of-variables formula for the last two integrals in~\eqref{eq:plug_back_in} to get 
	\begin{align*}
	f(X_t)&=f(X_0)+\widetilde{b}^i\int_0^t V^L_i f(X_{s-})\d s +\sigma_j^i\int_0^t V^L_i f(X_{s-})\circ \d B^j_s \\
	&+\int_0^t\int_{G\backslash\{e\}}\left( f(X_{s-}h)-f(X_{s-})\right)\widetilde{N}'(\d s,\d h) \\
	&+ \int_0^t\int_{G\backslash\{e\}}\left( f(X_{s-}h)-f(X_{s-})-\xi^i(h)V^L_i f(X_{s-})\right)\d s\, \mu(\d h).
	\end{align*}
	Technically, after using a change-of-variables formula we get integration against measure $\widetilde{\mu}:=\widetilde{\exp}_*\nu$ and the Poisson random measure corresponding to that measure with the domain of the integration being the whole group $G$, but since
	both integrands vanish at $e$, we can instead use $\mu$ and $\widetilde{N}'$, and the domain of integration reduces to $G\backslash\{e\}.$ This last stochastic integral equation for $X$ is also known as a \levy{}-It\^o decomposition
	of a \levy{} process in a Lie group (see \cite[Thm.~1.2]{liao_levy}) and it indeed follows that $X$ is a (left) \levy{} process with a \levy{} measure $\mu$.
\end{proof}

\begin{proof}[Proof of Proposition~\ref{prop:reprLevyOnLieViaHor}]
	The \textbf{only if} part is a consequence of Proposition~\ref{prop:sol_is_levy} so we now focus on the \textbf{if} part.
	The generator of a (left) \levy{} process is identified in \cite[Thm.~1.1]{liao_levy}. If we are able to find a \levy{} measure $\nu$ on $\RR^d$ such that
	$\mu=\widetilde{\exp}_*\nu|_{G\backslash\{e\}}$ holds, we may then reverse the computations in the proof of Proposition~\ref{prop:sol_is_levy} to express the generator of a left \levy{} process in a form on which we may 
	use the	characterization from Theorem~\ref{thm:charLevyManifolds} and thus find a required process $Y$.
	
	To construct such a suitable measure $\nu$ we first take open sets $U\subseteq G$ and $O\subseteq\RR^d$ such that $\widetilde{\exp}\colon O\to U$ is a diffeomorphism. This
	allows us to define a measure $\rho_O$ on $O$ by $\rho_O=\widetilde{\exp}^{-1}_*(\mu|_U).$ Since $\mu(\mathrm{Im}(\exp)^c)=0,$ we can treat $\mu$ as a measure on $\mathrm{Im}(\exp)$.  We denote
	$F= \widetilde{\exp}^{-1}(U^c)$ and then the restricted map $\widetilde{\exp}\colon F\to \mathrm{Im}(\exp)\cap U^c$ is
	still surjective. Since $\mu$ is a finite measure on $\mathrm{Im}(\exp)\cap U^c$ we can then use \cite[Prop.~1.101]{doberkat2007stochastic} to show that there exists a measure $\rho_F$ on $F$ such that
	$\widetilde{\exp}_*\rho_F=\mu|_{\mathrm{Im}(\exp)\cap U^c}.$ Actually, \cite{doberkat2007stochastic} deals with sub-probability measures but all of the results extend to finite measures. Finally, the measure $\nu$ on $\RR^d$ is
	defined by $\nu(A)=\rho_O(A\cap O)+\rho_F(A\cap F)$ for a measurable set $A$. Clearly $\mu=\widetilde{\exp}_*\nu|_{G\backslash\{e\}}$ holds and, moreover, $\nu$ is a \levy{} measure since
	$\nu(\{0\})=\mu(\{e\})=0, \nu(O^c)=\mu(U^c)<\infty$ and $\int_{O}\sum_{i=1}^d (x^i)^2\nu(\d x)=\int_U\sum_{i=1}^d(\xi^i(h))^2\mu(\d h)<\infty,$ where $\xi^i$ are coordinate functions introduced in the proof of Proposition~\ref{prop:sol_is_levy}.  
\end{proof}

\section*{Acknowledgements}
AM and VM are supported by The Alan Turing Institute under the EPSRC grant EP/N510129/1;
AM supported by EPSRC grant EP/P003818/1 and the Turing Fellowship funded by the Programme
on Data-Centric Engineering of Lloyd’s Register Foundation; VM supported by the PhD scholarship of Department of Statistics, University of Warwick.

\appendix

\section{Connections, geodesics, horizontal lift and anti-development} \label{app:connections}

A \emph{(linear) connection}  is given in terms of a covariant derivative, which prescribes how to derive vector fields along each other. A \emph{covariant derivative} is a bilinear map $\nabla\colon \Gamma(TM) \times \Gamma(TM) \to \Gamma(TM) $, which is additionally $C^\infty(M)$-linear in the first argument and obeys
a product rule in the second argument i.e.\ $\nabla_X(fY)=f\nabla_X Y+(Xf)Y.$ Covariant derivatives act locally; to compute $(\nabla_X Y)_p$ it is enough to only know the value $X_p$ and the values of $Y$
along a curve through $p$ which has a tangent vector $X_p$ at the point $p$. In local coordinates $x^i,$ the vector fields
$\partial_i:=\frac{\partial}{\partial x^i}$ form a local frame, i.e.\ the tangent vectors $\partial_i|_p$ form a basis of the tangent space $T_pM$ at any point $p$ in the local chart. Hence, we can locally express
$\nabla_{\partial_i}\partial_j=\Gamma^k_{ij}\partial_k$ for some smooth functions $\Gamma^k_{ij}$, which are known as Christoffel symbols. Christoffel symbols determine the covariant derivative. 

A vector field $\X$ along a curve $\gamma$, is said to be \emph{parallel (along the curve $\gamma$)} if $\nabla_{\dot{\gamma}}\X=0$ holds at every point of the curve. In a local chart a vector field along $\gamma$
can be expressed as $\X_{\gamma_t}=\X^k(t)\partial_k$ and a parallel vector field is then a solution of a particular system of ordinary differential equations (ODEs), namely
$$ \dot{\X}^k(t)+\Gamma^k_{ij}(\gamma_t)\dot{\gamma}^i_t\X^j(t)=0.$$
Once a curve $\gamma$ with $\gamma_0=p$ and $v\in
T_{p}M$ are fixed, the local existence and uniqueness theorem for ODEs guarantees that there is a unique parallel vector field $\X$ along $\gamma$ with $\X_{p}=v$. Since the system of ODEs is (non-autonomous) linear, the solution cannot explode and is defined along the whole curve $\gamma$.

Actually, we have
constructed a unique (parallel) lift of a curve onto a tangent bundle (once the initial point is fixed).
Additionally, using parallel vector fields along the curve $\gamma$ with $\gamma_0=p$ and $\gamma_1=q$ we can connect the tangent spaces $T_{p}M$ and $T_{q}M$. More precisely, for each $v\in T_{p}M$ we define
$\tau_\gamma(v)\in T_{q}M$ as a value at $q$ of the unique parallel vector field $\X$ (along $\gamma$) with
$\X_{p}=v.$ The map $\tau_\gamma$ is called the \emph{parallel transport along $\gamma$} and can be shown to be a linear isomorphism.

Covariant derivative  also allows us to define \emph{geodesics}, which are special curves $\gamma$ such that their velocity vector field $\dot{\gamma}$ is parallel along $\gamma$. A geodesic is locally given as a solution of the system of second order ODEs
\begin{equation} \label{eq:geodesic}
	\ddot{\gamma}^k_t+\Gamma^k_{ij}(\gamma_t)\dot{\gamma}^i_t\dot{\gamma}^j_t=0
\end{equation} 
for $k=1,\ldots,d$
and therefore a geodesic is unique once we fix an initial point and an initial velocity. This allows us to define an exponential map in
essentially the same manner as on Riemannian manifolds. A \emph{geodesic exponential map} is a map $\Exp_p\colon O\subseteq T_pM\to M, \Exp_p(X):=\gamma(1),$ where $\gamma$ is a unique geodesic with $\gamma(0)=p$ and
$\dot{\gamma}(0)=X$ and $O$ is an open star-shaped neighbourhood of $0\in T_pM$
where the geodesic exponential map is well defined\footnote{The differential
equation for a geodesic has a unique solution defined on a maximal time
interval. However, this interval might not be all of $\RR$. Hence geodesics
might explode in finite time and
the domain of the geodesic exponential map might have to be restricted.}. 
If the domain of the geodesic exponential map at $p$ is equal to $T_p$ for any $p\in M$ we say that $M$ is \emph{geodesically complete}.
Note that by Hopf-Rinow theorem compact Riemannian manifolds are always geodesically complete, but there exist non-metric connections on compact manifolds which are not geodesically complete~\cite[Ex.~7.16]{OneillRelativity}.

In Section~\ref{generalLevy} we introduced a frame bundle. On the frame bundle a notion of \emph{vertical}
vectors is naturally defined. More precisely, the vertical space is defined by $V_uF(M):=\ker d\pi_u.$ We now show how
the covariant derivative can be used to define horizontal vectors and get a splitting required for the connection. Let $u_t$ be a
smooth curve in $F(M)$ so that $\pi(u_t)$ is a smooth curve in $M$. A curve $u_t$ is called \emph{horizontal} if a vector field $u_t e$ is parallel along the curve $\pi(u_t)$ for each $e \in \RR^d.$ Additionally,  a vector in $T_uF(M)$ is called \emph{horizontal} if it is a velocity vector of a horizontal curve through $u$. The space of horizontal vectors at
$u$ is denoted by $H_uF(M).$  It can be easily checked that vertical and horizontal space intersect trivially and furthermore the equality $T_uF(M)=V_uF(M)\oplus H_uF(M)$ holds, giving us the required splitting
of tangent spaces of the frame bundle\footnote{This kind of splitting is also known as a principal Ehrenmanns connection on a principal fibre bundle, whereas covariant derivative is also known as a linear connection and can be thought
as a structure on the tangent bundle $TM$. Since $TM$ is an associated bundle of $F(M)$, (principal) connections on $F(M)$ are naturally induced by linear connections on $TM$ and vice versa \cite[Thm.~7.5 in Ch.~III]{kobayashinomizu}.}. It is also easily seen
that the following equivariance property holds: $d (R_g)_u(H_uF(M))=H_{ug}F(M)$ for all $g\in\GL(d),$ where $d(R_g)_u\colon T_uF(M)\to T_{ug}F(M)$ is a differential at $u$ of a (right) group action map $R_g\colon u \mapsto ug.$ 

Since we have the splitting, dimensionality arguments show that the differential $d\pi_u$ restricted to the horizontal space is an isomorphism between $H_uF(M)$ and $T_{\pi(u)}M.$ The inverse of this (restricted) differential
defines a \emph{horizontal lift} of tangent vectors in $M$. Moreover, horizontal lifts can be obtained for (piecewise) smooth curves. Let $\gamma$ be a smooth curve in $M$. We call a
curve $u_t$ in $F(M)$ a \emph{horizontal lift of curve $\gamma$} if $\pi(u_t)=\gamma_t$ and it is horizontal, i.e.\ $\dot{u}_s\in H_{u_s}F(M)$ for all $s$. Once an initial frame $u_0\in\pi^{-1}(\{\gamma_0\})$ is fixed, the horizontal lift is unique and exists for all
times for which the underlying curve $\gamma$ is defined. This can be seen by the following computation in a local chart. A local chart for $O\subseteq M$ around $p$ induces  a local chart for $\pi^{-1}(O)\subseteq F(M)$ around
$\pi^{-1}(p)$. A frame $u$ above $p$ can be uniquely represented as $u=(x^i,r_m^{k})$ where $ue_m=r_m^{k}\partial_k$. Then the horizontal lift of $\gamma$ is given as 
$(\gamma^i_t,r_m^{k}(t)),$ where $r^k_m$ satisfy a system of linear ODEs 
\begin{equation} \label{horliftlocal}
\dot{r}^k_m(t)+r^l_m(t)\Gamma^k_{jl}(\gamma_t)\dot{\gamma}^j_t =0.
\end{equation}
Uniqueness and existence of the horizontal lift after fixing the initial frame thus follows from the theory of ODEs and linearity of the system implies non-explosion. 

Furthermore, if $u_t$ is a horizontal lift of $\gamma$ with an initial frame $u_0$, then $u_tg$ is a horizontal lift of $\gamma$ with an initial frame $u_0g$ for any $g\in\GL(d).$ In particular, we
recover parallel transport as $\tau_\gamma=u_1u_0^{-1}$ for any horizontal lift $u_t$ of $\gamma$, since the quantity on the right hand side does not depend on the initial frame (and hence the particular horizontal lift).

Horizontal spaces can also be described by a family of horizontal vector fields. 
For any $e\in\RR^d$ a \emph{fundamental horizontal vector field $H_e$} is defined by  $H_e(u):=(d\pi_u)^{-1}(ue)\in H_uF(M)$ for any frame $u$. Since the definition is pointwise, the smoothness is
not immediate, but quickly follows from the following representation in local coordinates taken from \cite[Prop.~2.1.3]{hsumanifolds}.

\begin{lemma} \label{lemma:horInlocal}
	In terms of a local chart on $F(M)$ where $u=(x^i,r^k_m)$, we have 
	$$H_e(u)=e^ir^j_i\frac{\partial}{\partial x^j}-e^ir^j_ir^l_m\Gamma^k_{jl}(x)\frac{\partial}{\partial r^k_m},$$
	where $u=(x,r)=(x^i,r^k_m)\in F(M)$ and $e=(e^i).$
\end{lemma}

It is trivially seen that the map $e \mapsto H_e$ is linear. Let
$e_1,\ldots,e_d$ denote the standard base of $\RR^d$, then $H_i:=H_{e_i}$ and vector fields $H_1,\ldots, H_d$ span the whole horizontal space $H_uF(M)$ at any frame $u$. The following lemma, taken from \cite[Prop.~III.2.2]{kobayashinomizu}, relates the fundamental horizontal fields at different frames above the same point in the manifold.

\begin{lemma} \label{lemma:changeHorizontal}
	Let $e\in \RR^d$ and $g\in \GL(d).$ Then $$H_{e}(ug)=d (R_g)_u(H_{ge}(u))$$ for every $u\in F(M),$ where
	$d(R_g)_u\colon T_uF(M)\to T_{ug}F(M)$ is a differential at $u$ of a (right) group action map $R_g\colon u \mapsto ug.$ 
\end{lemma}
The following two lemmas are easy corollaries.
\begin{lemma}  \label{lemma:changeexponential}
	For any $u\in F(M),\, e\in \RR^d,\, g\in \GL(d)$ we have
	$\exp(H_e)(ug)=\exp(H_{ge})(u)\cdot g$.
\end{lemma}
\begin{proof}
	If $\psi$ is an integral curve of $H_{ge}$ with $\psi(0)=u$, then Lemma~\ref{lemma:changeHorizontal} implies that $\widetilde{\psi}:=\psi g$ is an integral curve of
	$H_{e}$ with $\widetilde{\psi}(0)=ug$ and the claim follows from the uniqueness of integral curves and the definition of the flow exponential map.
\end{proof}
\begin{lemma}  \label{lemma:changeiterated}
	For every smooth function $f\in C^\infty(F(M)), x_1,\ldots,x_m \in \RR^d$ and $g\in \GL(d)$ we have
	$$(H_{x_1}\cdots H_{x_m} f)\circ R_g=H_{gx_1}\cdots H_{gx_m}(f\circ R_g). $$
\end{lemma}
\begin{proof}
	We prove the lemma by induction on $m$. The base case of $m=1$ is just a restatement of Lemma~\ref{lemma:changeHorizontal}. For an inductive step let the claim hold for $m$ vectors in $\RR^d$ and
	take an additional vector $x_{m+1}$. Denote $g=H_{x_{m+1}}f$ and then
	$$ (H_{x_1}\cdots H_{x_{m+1}} f)\circ R_g=(H_{x_1}\cdots H_{x_m} g)\circ R_g=H_{gx_1}\cdots H_{gx_m}(g\circ R_g)=H_{gx_1}\cdots H_{gx_{m+1}}(f\circ R_g),$$ where the second equality holds by the induction hypothesis and the last equality by the base case $m=1.$		
\end{proof}

So far we were able to start with a curve $\gamma$ in $M$ and after fixing an initial frame $u_0$ we could construct a unique horizontal lift $u_t$ in the frame bundle $F(M)$. We can go one step further and define an \emph{anti-development of a curve $\gamma$} (or of its lift $u_t$) by 
\begin{equation*}
w_t:=\int_0^tu_s^{-1}(\dot{\gamma}_s)\d s,
\end{equation*}
which is a smooth curve in $\RR^d$ started at $0$. One could also reverse the construction and start with a curve $w_t$ in $\RR^d$ and construct a horizontal curve $u_t$ on $F(M)$ (and a curve
$\gamma_t=\pi(u_t)$ on $M$) whose anti-development is $w_t.$ Once we fix the initial frame this correspondence is unique and $u_t$ is given as a solution of the ODE 
\begin{equation*}
\dot{u}_t=H_i(u_t)\dot{w}^i_t.
\end{equation*}
Therefore, we obtain a one-to-one correspondence between a (piecewise) smooth curves in $M$, their horizontal lifts in $F(M)$ and their anti-developments in $\RR^d$ once the initial frame is fixed.

It is interesting to check how this correspondence is expressed for special curves -- geodesics. One quickly
finds that the anti-development of a geodesic is a uniform motion, i.e.\ a constant speed straight line. This simplifies
the ODE that the horizontal lift $u_t$ is a solution of and actually shows that $u_t$ is an integral curve of some fundamental horizontal vector field, i.e.\ there exists some $e\in\RR^d$ such that $\dot{u}_t=H_e(u_t).$ Therefore,
 uniform motion in $\RR^d$ corresponds to moving along an integral curve of some fundamental horizontal vector field or equivalently moving along a geodesic in $M.$ This also relates flow exponential map\footnote{We have
 only defined a flow of complete fields, but analogously we could define a local flow of a (non-complete) field.} on $F(M)$ to a geodesic exponential map on $M$ so that the equality 
\begin{equation} \label{bothExponentials}
\pi(\exp(H_e)(u))=\Exp_p(d\pi_u(H_e))=\Exp_p(ue)
\end{equation} holds for any $e\in \RR^d$ and any frame $u$ with $\pi(u)=p.$

\bibliographystyle{amsalpha}
\bibliography{source}

\providecommand{\bysame}{\leavevmode\hbox to3em{\hrulefill}\thinspace}
\providecommand{\MR}{\relax\ifhmode\unskip\space\fi MR }
\providecommand{\MRhref}[2]{%
  \href{http://www.ams.org/mathscinet-getitem?mr=#1}{#2}
}
\providecommand{\href}[2]{#2}
\begin{thebibliography}{{\c{C}}JPS80}

\bibitem[AE00]{isotropiclevy}
D.~Applebaum and A.~Estrade, \emph{Isotropic {L{\'e}vy} processes on
  {Riemannian} manifolds}, The Annals of Probability \textbf{28} (2000), no.~1,
  166--184.

\bibitem[AK93]{ApplebaumKunitaLevyFlow}
D.~Applebaum and H.~Kunita, \emph{L\'evy flows on manifolds and {L\'evy}
  processes on {Lie} groups}, J. Math. Kyoto Univ. \textbf{33} (1993), no.~4,
  1103--1123.

\bibitem[AL]{ApplebaumLiaoMarkovJumps}
D.~Applebaum and M.~Liao, \emph{{Markov Processes with Jumps on Manifolds and
  Lie Groups}}, Geometry and Invariance in Stochastic Dynamics, to appear.

\bibitem[App95]{ApplebaumHorizontalFirst}
D.~Applebaum, \emph{A horizontal {L\'evy} process on the bundle of orthonormal
  frames over a complete {Riemannian} manifold}, S\'eminaire de probabilit\'es
  de Strasbourg \textbf{29} (1995), 166--180 (en). \MR{1459458}

\bibitem[App04]{AppLevyBook}
\bysame, \emph{{L{\'e}vy} {Processes} and {Stochastic} {Calculus}}, Cambridge
  Studies in Advanced Mathematics, Cambridge University Press, 2004.

\bibitem[ASB20]{ApplebaumRosieL2}
D.~Applebaum and R.~Shewell~Brockway, \emph{{$L^2$} properties of {L\'evy}
  {Generators} on {Compact} {Riemannian} {Manifolds}}, Journal of Theoretical
  Probability (2020).

\bibitem[Ber96]{BertoinLevy}
J.~Bertoin, \emph{{L{\'e}vy} {Processes}}, Cambridge Tracts in Mathematics 121,
  Cambridge University Press, 1996.

\bibitem[{\c{C}}J81]{CJrepr}
E.~{\c{C}}inlar and J.~Jacod, \emph{Representation of {Semimartingale} {Markov}
  {Processes} in {Terms} of {Wiener} {Processes} and {Poisson} {Random}
  {Measures}}, Seminar on Stochastic Processes, 1981 (E.~{\c{C}}inlar, K.~L.
  Chung, and R.~K. Getoor, eds.), Birkh{\"a}user Boston, 1981, pp.~159--242.

\bibitem[{\c{C}}JPS80]{CJPSMarkovSemi}
E.~{\c{C}}inlar, J.~Jacod, Philip Protter, and M.~Sharpe, \emph{Semimartingales
  and {Markov} processes}, Probability Theory and Related Fields \textbf{54}
  (1980), 161--219.

\bibitem[Coh96a]{CohenSDEs1}
S.~Cohen, \emph{G\'eom\'etrie diff\'erentielle stochastique avec sauts 1},
  Stochastics and Stochastic Reports \textbf{56} (1996), no.~3-4, 179--203.

\bibitem[Coh96b]{CohenSDEs2}
\bysame, \emph{G\'eom\'etrie diff\'erentielle stochastique avec sauts 2:
  discr\'etisation et applications des eds avec sauts}, Stochastics and
  Stochastic Reports \textbf{56} (1996), no.~3-4, 205--225.

\bibitem[Dob07]{doberkat2007stochastic}
E.E. Doberkat, \emph{Stochastic {Relations}: {Foundations for Markov Transition
  Systems}}, Chapman \& Hall/CRC Studies in Informatics Series, CRC Press,
  2007.

\bibitem[Dyn65]{DynkinBook}
E.~B. Dynkin, \emph{{M}arkov processes: Volume 1}, 1 ed., Die Grundlehren der
  Mathematischen Wissenschaften 121/122, Springer-Verlag Berlin Heidelberg,
  1965.

\bibitem[Est97]{EstradeMarkov}
A.~Estrade, \emph{A characterization of {M}arkov solutions for stochastic
  differential equations with jumps}, S\'eminaire de probabilit\'es de
  Strasbourg \textbf{31} (1997), 315--321 (en). \MR{1478740}

\bibitem[Fuj91]{fujiwaraLevyFlowManifolds}
T.~Fujiwara, \emph{Stochastic differential equations of jump type on manifolds
  and {L\'evy} flows}, J. Math. Kyoto Univ. \textbf{31} (1991), no.~1, 99--119.

\bibitem[GS72]{GetoorSharpeConformalMartingales}
R.~K. Getoor and M.~J. Sharpe, \emph{Conformal martingales}, Inventiones
  Mathematicae \textbf{16} (1972), no.~4, 271--308.

\bibitem[HO56]{HanoOzekiLinearConnections}
J.~Hano and H.~Ozeki, \emph{On the holonomy groups of linear connections},
  Nagoya Math. J. \textbf{10} (1956), 97--100.

\bibitem[Hsu02]{hsumanifolds}
E.P. Hsu, \emph{{Stochastic Analysis on Manifolds}}, Graduate Studies in
  Mathematics 38, vol.~38, American Mathematical Society, 2002.

\bibitem[Hun56]{huntgroups}
G.~A. Hunt, \emph{{Semi-Groups of Measures on {Lie} Groups}}, Transactions of
  the American Mathematical Society \textbf{81} (1956), no.~2, 264--293.

\bibitem[JP91]{JacodProtterMarkov}
J.~Jacod and P.~Protter, \emph{Une remarque sur les equations differentielles
  stochastiques a solutions markoviennes}, S{\'e}minaire de Probabilit{\'e}s
  XXV (Berlin, Heidelberg) (Jaques Az{\'e}ma, Marc Yor, and Paul~Andr{\'e}
  Meyer, eds.), Springer Berlin Heidelberg, 1991, pp.~138--139.

\bibitem[JS03]{JacodShiryaevLimit}
J.~Jacod and A.~N. Shiryaev, \emph{{Limit Theorems for Stochastic Processes}},
  2 ed., Grundlehren der mathematischen Wissenschaften 288, Springer-Verlag
  Berlin Heidelberg, 2003.

\bibitem[Kal90]{KallenbergMarked1990}
O.~Kallenberg, \emph{Random time change and an integral representation for
  marked stopping times}, Probability Theory and Related Fields \textbf{86}
  (1990), no.~2, 167--202.

\bibitem[KN63]{kobayashinomizu}
S.~Kobayashi and K.~Nomizu, \emph{{Foundations of Differential Geometry}},
  Wiley Classics Library, vol.~1, Wiley-Interscience, 1963.

\bibitem[KPP95]{genStratonovich}
T.~G. Kurtz, \'E. Pardoux, and P.~Protter, \emph{Stratonovich stochastic
  differential equations driven by general semimartingales}, Annales de
  l'I.H.P. Probabilit\'es et Statistiques \textbf{31} (1995), no.~2, 351--377
  (en). \MR{1324812}

\bibitem[Kun19]{KunitaBookStochFlows}
H.~Kunita, \emph{Stochastic flows and jump-diffusions}, 1 ed., Probability
  Theory and Stochastic Modelling, Springer, 2019.

\bibitem[Kyp14]{kyprianouFluctuationsLevy}
A.E. Kyprianou, \emph{{Fluctuations of L{\'e}vy Processes with Applications:
  Introductory Lectures}}, Universitext, Springer Berlin Heidelberg, 2014.

\bibitem[Lee12]{leeIntro}
J.M. Lee, \emph{{Introduction to Smooth Manifolds}}, 2nd ed., Graduate Texts in
  Mathematics 218, Springer-Verlag New York, 2012.

\bibitem[Lia04]{liao_levy}
Ming Liao, \emph{{L{\'e}vy} processes in {Lie} groups}, Cambridge Tracts in
  Mathematics, Cambridge University Press, 2004.

\bibitem[Mar81]{marcusSDE}
S.~I. Marcus, \emph{Modeling and approximation of stochastic differential
  equations driven by semimartingales}, Stochastics \textbf{4} (1981), no.~3,
  223--245.

\bibitem[Moh04]{mohariReduction}
A.~Mohari, \emph{{Ergodicity of {L{\'e}vy} flows}}, Stochastic Processes and
  their Applications \textbf{112} (2004), no.~2, 245--259.

\bibitem[O'N83]{OneillRelativity}
B.~O'Neill, \emph{Semi-{R}iemannian geometry: with applications to relativity},
  Pure and applied mathematics 103, Academic Press, 1983.

\bibitem[PE92]{horizontalLift}
M.~Pontier and A.~Estrade, \emph{Rel\`evement horizontal d'une semimartingale
  c\`adl\`ag}, S\'eminaire de probabilit\'es de Strasbourg \textbf{26} (1992),
  127--145 (fr). \MR{1231989}

\bibitem[Pro92]{ProtterIntegration}
P.~Protter, \emph{{Stochastic Integration and Differential Equation}}, second
  ed., Springer-Verlag, Berlin, Heidelberg, 1992.

\bibitem[Rog81]{rogersonSDEs}
S.~J. Rogerson, \emph{Stochastic dynamical systems and processes with
  discontinuous sample paths}, Ph.D. thesis, University of Warwick, 1981.

\bibitem[RS17]{DifGoMathPhys}
G.~Rudolph and M.~Schmidt, \emph{{Differential Geometry and Mathematical
  Physics: Part II. Fibre Bundles, Topology and Gauge Fields}}, 1st ed.,
  Theoretical and Mathematical Physics, Springer, 2017.

\bibitem[RY99]{revuzyor}
D.~Revuz and M.~Yor, \emph{{Continuous Martingales and {B}rownian Motion }},
  3rd ed., Springer, 1999.

\bibitem[Sat99]{satolevy}
K.~Sato, \emph{{{L{\'e}vy} Processes and Infinitely Divisible Distributions}},
  Cambridge Studies in Advanced Mathematics, Cambridge University Press, 1999.

\bibitem[Sch12]{SchnurrSemimartingaleKilling}
Alexander Schnurr, \emph{{On the semimartingale nature of Feller processes with
  killing}}, Stochastic Processes and their Applications \textbf{122} (2012),
  no.~7, 2758--2780.

\end{thebibliography}

\end{document}